\documentclass{article}

\usepackage[ngerman,english]{babel}

\usepackage[T1]{fontenc}
\usepackage{lmodern}
\usepackage{parskip}

\usepackage{amsfonts}
\usepackage{amssymb}
\usepackage{amsthm}
\usepackage{amsbsy}
\usepackage{latexsym}
\usepackage{bm}

\usepackage{subfigure}
\usepackage{enumerate}

\usepackage{graphicx}
\usepackage[dvipsnames]{xcolor}
\usepackage{epsfig}

\usepackage{textcomp}

\usepackage{float}

\usepackage[latin1]{inputenc}
\usepackage[d]{esvect} 
\DeclareMathAlphabet{\mathpzc}{OT1}{pzc}{m}{it}
\usepackage{upgreek}
\usepackage{mathtools}
\usepackage[mathscr]{eucal}
\usepackage{ushort}
\usepackage{array}
\usepackage{vmargin}

\newtheorem{theorem}{Theorem}
\newtheorem{lemma}{Lemma}
\newtheorem{corollary}{Corollary}
\newtheorem{remark}{Remark}

\title{Monotonicity-Based Regularization for Shape Reconstruction in Linear Elasticity}
\author{Sarah Eberle\thanks{eberle@math.uni-frankfurt.de, Institute of Mathematics,
Goethe-University Frankfurt, Frankfurt am Main, Germany (corresponding author)}
\and Bastian Harrach\thanks{harrach@math.uni-frankfurt.de, Institute of Mathematics,
Goethe-University Frankfurt, Frankfurt am Main, Germany}}

\date{}

\begin{document}

\maketitle

\begin{abstract}
\noindent
{\color{black}
We deal with the shape reconstruction of inclusions in elastic bodies. For solving this inverse problem in practice, data fitting functionals are used. 
Those work better than the rigorous monotonicity methods from \cite{Eberle_mon_test}, but have no rigorously proven convergence theory. Therefore we show how the monotonicity
methods can be converted into a regularization method for a data-fitting functional without losing the convergence properties of the monotonicity methods.
This is a great advantage and a significant improvement over standard regularization techniques. 
In more detail, we introduce constraints on the minimization problem of the residual based on the monotonicity methods and prove the existence and
uniqueness of a minimizer as well as the convergence of the method for noisy data. In addition, we compare numerical reconstructions of inclusions 
based on the monotonicity-based regularization with a standard approach (one-step linearization with Tikhonov-like regularization), which also shows
the robustness of our method regarding noise in practice.
}
\end{abstract}



\section{Introduction} 

The main motivation is the non-destructive testing of elastic structures, such as is required for material examinations, 
in exploration geophysics, and for medical diagnostics (elastography). 
From a mathematical point of view, this constitutes an inverse problem since we have only measurement data on the boundary and
not inside of the elastic body.
This problem is highly ill-posed, since even the smallest measurement errors can completely falsify the result.
\\
\\
There are several authors who deal with the theory of the inverse problem of elasticity.
For the two dimensional case, we refer the reader to  \cite{ikehata1990inversion,nakamura1993identification,
imanuvilov2011reconstruction,
lin2017boundary}. In three dimensions, \cite{nakamura1995inverse, nakamura2003global} and 
\cite{eskin2002inverse}
gave the proof for uniqueness results for both Lam\'e coefficients under the assumption that $\mu$ 
is close to a positive constant. 
\cite{beretta2014lipschitz, beretta2014uniqueness} proved the uniqueness for partial data, where the
Lam\'e parameters are piecewise
constant and some boundary determination results were shown in \cite{nakamura1999layer,nakamura1995inverse,lin2017boundary}.
 \\
 \\
Further on, solution methods applied so far for the inverse problem, which will be solved in this paper,
were presented in the following works:
In \cite{Oberai_2004} and \cite{Oberai_2003}, the time-independent inverse problem of linear elasticity is
solved by means of the adjoint method and the reconstruction is simulated numerically. In addition, \cite{Seidl}
deals with the coupling of the state and adjoint equation and added 
two variants of residual-based stabilization
to solve the inverse linear elasticity problem for incompressible plane stress.
A boundary element-Landweber method for the Cauchy problem in  
stationary linear elasticity was investigated in \cite{Marin_2005}. 
In \cite{Hubmer}, the stationary inverse problem was solved by means of a Landweber iteration as well and numerical 
examples were presented. Reciprocity principles for the detection of cracks in elastic bodies were investigated, 
for example, in \cite{Andrieux} and \cite{Steinhorst} or more recently in \cite{Ferrier}. 
By means of a regularization approach, a stationary elastic inverse problem is solved in \cite{Jadamba} and applied in
numerical examples. \cite{Marin} introduces a regularized boundary element method. 
Finally, we want to mention the monotonicity methods for linear elasticity developed by the authors of this
paper in \cite{Eberle_mon_test} {\color{black} as well as its application for the reconstruction of inclusions based on
experimental data in \cite{Eberle_Experimental}}.
\noindent
\\
\\
{\color{black}
We want to point out that the reconstruction of the support of the Lam\'e parameters, also called shape in this paper, and not the reconstruction of their values is the topic of this work.} The key issue of the shape reconstruction of inclusions is the monotonicity 
property of the corresponding Neumann-to-Dirichlet operator 
(see \cite{Tamburrino06, tamburrino2002new}).
These monotonicity properties were also applied for the construction of monotonicity tests for electrical impedance 
\mbox{tomography} (EIT), e.g., in \cite{harrach2013monotonicity}, as well as the monotonicity-based regularization
in \cite{harrach2016enhancing}. In practice however, data fitting functionals provide better results than the 
monotonicity methods but {\color{black}the data-fitting functionals are usually not convex (see, e.g. \cite{Harrach_FEM}).
Even for exact data, therefore, it cannot generally be guaranteed that the algorithm does not erroneously deliver a local minimum. In addition, there is noise and ill-posedness. The local convergence theory of Newton-like methods requires non-linearity assumptions such as the tangential cone condition, which are still not proven even for simpler examples such as EIT. The convergence theory of Tikhonov-regularized data fitting functionals applies to their global minima, which in general cannot be found due to the non-convexity. Our method is based on the minimization of a convex functional and is to the knowledge of the authors the first rigorously convergent method for this problem, but only provides the shape of the inclusions.}
We combine the monotonicity methods 
(cf. \cite{Eberle_Monotonicity} and \cite{Eberle_mon_test})
with data fitting functionals to obtain convergence results and an improvement of both methods regarding stability
for noisy data.
{\color{black}
Here, we want to remark that compared to other data-fitting methods, we use the following a-priori assumptions: 
the Lam\'e parameters fulfill monotonicity relations, have a common support, the lower and upper bounds of the contrasts of the anomalies are known and we deal with a constant and known background material.
}
Compared with \cite{harrach2016enhancing}, we expand the approach used there from the consideration of only one parameter
to two parameters. 
\\
\\
The outline of the paper is as follows: 
We start with the introduction of the problem statement.
In order to detect and reconstruct inclusions in elastic bodies, we aim to determine the difference between an unknown
Lam\'e parameter pair $(\lambda,\mu)$ and that of the known background  $(\lambda_0,\mu_0)$ and formulate a minimization problem.
Similar to the linearized monotonicity tests in \cite{Eberle_mon_test}, we also consider the Fr\'echet derivative,
which approximates the difference between two Neumann-to-Dirichlet operators.
For solving the resulting minimization problem, we first take a look at a standard approach (standard one-step linearization method). 
Therefore regularization parameters are introduced, which can only be determined heuristically. For this purpose, for example, a parameter
study can be carried out. We would like to point out that this method is only a heuristic approach, but is commonly used in practice.
Overall, this heuristic approach leads to reconstructions of the unknown inclusions 
without a rigorous theory. 
In Section 4, we focus on the monotonicity-based regularization in order to enhance the data fitting functionals.
The idea of the regularization is to introduce conditions for the parameters / inclusions to
be reconstructed for the
minimization problem, which are based on the monotonicity properties of the Neumann-to-Dirichlet operator and 
the monotonicity tests. Further on, we prove that there exists a unique minimizer for this problem and 
that we obtain convergence even for noisy data. Finally, we compare numerical reconstructions of inclusions 
based on the monotonicity-based regularization with the one-step linearization, which also shows
the robustness of our method regarding noise in practice.

\section{Problem Statement}
We start with the introduction of the problems of interest, e.g., the {\it direct} as well as {\it inverse problem}
of stationary linear elasticity.
\\
\noindent
Let $\Omega\subset \mathbb{R}^d$ ($d=2$ or $3$) be a bounded and connected open set 
with Lipschitz boundary $\partial\Omega=\Gamma=\overline{\Gamma_{\mathrm{D}}\cup \Gamma_{\mathrm{N}}}$, $\Gamma_{\mathrm{D}}\cap \Gamma_{\mathrm{N}}=\emptyset$,
where $\Gamma_{\textup D}$ and $\Gamma_{\textup N}$ are the corresponding Dirichlet and Neumann boundaries.
{\color{black} We assume that  $\Gamma_{\textup D}$ and $\Gamma_{\textup N}$ are relatively open and connected.}
{\color{black}
For the following, we define
\begin{align*} 
L_+^\infty(\Omega):=\lbrace w \in L^\infty(\Omega):\underset{x\in\Omega}{\text{ess\,inf}}\,w(x)>0\rbrace.
\end{align*}
}
\noindent
Let $u:\Omega\to\mathbb{R}^d$ be the displacement vector, $\mu,\lambda:\Omega\to L^{\infty}_+(\Omega)$ the Lam\'{e} parameters, 
$\hat{\nabla} u=\tfrac{1}{2}\left(\nabla u + (\nabla u)^T\right)$ the symmetric gradient, $n$ is \mbox{the normal}
vector pointing outside of $\Omega$ , $g\in L^{2}(\Gamma_{\textup N})^d$ the boundary force and $I$ the $d\times d$-identity matrix.
We define the divergence of a matrix $A\in \mathbb{R}^{d\times d}$ via 
$\nabla\cdot A=\sum\limits_{i,j=1}^d\dfrac{\partial A_{ij}}{\partial x_j}e_i$, where $e_i$ is a unit vector
and $x_j$ a component of a vector from $\mathbb{R}^d$.
\\
The boundary value problem of linear elasticity ({\it direct problem}) is 
{\color{black} that $u\in H^1(\Omega)^d$ solves}
\begin{align}\label{direct_1}
\nabla\cdot \left(\lambda (\nabla\cdot u)I + 2\mu \hat{\nabla} u \right)&=0 \quad \mathrm{in}\,\,\Omega,\\
\left(\lambda (\nabla\cdot u)I + 2\mu \hat{\nabla} u \right) n&=g\quad \mathrm{on}\,\, \Gamma_{\textup N},\label{direct_2}\\
u&=0 \quad \mathrm{on}\,\, \Gamma_{\textup D}.\label{direct_3}
\end{align}
\noindent
From a physical point of view, this means that we deal with an elastic test body $\Omega$ which is fixed (zero displacement)
at $\Gamma_{\mathrm{D}}$ (Dirichlet condition) and apply a force $g$ on $\Gamma_{\mathrm{N}}$ (Neumann condition).
This results in the displacement $u$, which is measured on the boundary $\Gamma_{\mathrm{N}}$.
\\
\\
\noindent
The equivalent weak formulation of the boundary value problem (\ref{direct_1})-(\ref{direct_3})  is
that $u\in\mathcal{V}$ fulfills
\begin{align}
\label{var-direct_1}
\int_{\Omega} 2 \mu\, \hat{\nabla}u : \hat{\nabla}v  + \lambda \nabla \cdot u \,\nabla\cdot  v\,dx=\int_{\Gamma_{\textup N}}g \cdot v \,ds \quad \text{ for all } v\in \mathcal{V},
\end{align}
where
$\mathcal{V}:=\left\{   v\in H^1(\Omega)^d:\,  v_{|_{\Gamma_{\textup D}}}=0\right\}$.
\\
\\
We want to remark that for $\lambda,\mu \in L^{\infty}_+(\Omega)$ the existence and uniqueness of a solution to the variational formulation (\ref{var-direct_1}) can be shown by
the Lax-Milgram theorem (see e.g., in \cite{Ciarlet}). 
\\
\\
Measuring boundary displacements that result from applying forces to $\Gamma_{\textup{N}}$ can be modeled by the
Neumann-to-Dirichlet operator $\Lambda(\lambda,\mu)$ defined by
\begin{align*}
\Lambda(\lambda,\mu): L^2(\Gamma_{\textup N})^d\rightarrow L^2(\Gamma_{\textup N})^d: \quad  g\mapsto u_{|_{\Gamma_{\textup N}}},
\end{align*}
\noindent
where $u\in\mathcal{V}$ solves (\ref{direct_1})-(\ref{direct_3}).
\\
\\
This operator is self-adjoint, compact and linear
{\color{black} (see Corollary 1.1 from \cite{Eberle_mon_test}}).
Its associated bilinear form is given by
\begin{align}
\langle g,\Lambda(\lambda,\mu)h\rangle=\int_{\Omega} 2 \mu\, \hat{\nabla}u^g_{(\lambda,\mu)} : \hat{\nabla}u^h_{(\lambda,\mu)}  + 
\lambda \nabla \cdot u^g_{(\lambda,\mu)} \,\nabla\cdot  u^h_{(\lambda,\mu)}\,dx,\label{bilinear_Lambda}
\end{align}
\noindent
where $u_{(\lambda,\mu)}^g$ solves the problem (\ref{direct_1})-(\ref{direct_3}) and $u_{(\lambda,\mu)}^h$ 
the corresponding problem with boundary force $h\in L^2(\Gamma_{\mathrm{N}})^d$.
\noindent
\\
\\
Another important property of $\Lambda(\lambda,\mu)$ is its Fr\'echet differentiability (for the corresponding proof 
see {\color{black} Lemma 2.3 in} \cite{Eberle_mon_test}).
For directions $\hat{\lambda}, \hat{\mu}\in L^\infty(\Omega)$, the derivative
\begin{align*}
\Lambda'(\lambda,\mu)(\hat{\lambda},\hat{\mu}):  L^2(\Gamma_{\textup N})^d\rightarrow  L^2(\Gamma_{\textup N})^d
\end{align*}
is the self-adjoint compact linear operator associated to the bilinear form
\begin{align*}
\langle \Lambda'(\lambda,\mu)(\hat{\lambda},\hat{\mu})g,h\rangle
=&-\int_{\Omega} 2 \hat{\mu}\, \hat{\nabla}u^g_{(\lambda,\mu)} : \hat{\nabla}u^h_{(\lambda,\mu)}  + \hat{\lambda} \nabla \cdot u^g_{(\lambda,\mu)} \,\nabla\cdot  u^h_{(\lambda,\mu)}\,dx.
\end{align*}
Note that for  $\hat{\lambda}_0, \hat{\lambda}_1, \hat{\mu}_0, \hat{\mu}_1 \in L^\infty(\Omega)$ with
$\hat{\lambda}_0\leq \hat{\lambda}_1 \text{ and }  \hat{\mu}_0 \leq \hat{\mu}_1$
we obviously have
\begin{align}\label{mon_Frechet}
\Lambda'(\lambda,\mu)(\hat{\lambda}_0,\hat{\mu}_0)\geq \Lambda'(\lambda,\mu)(\hat{\lambda}_1,\hat{\mu}_1),
\end{align}
\noindent
in the sense of {\color{black} quadratic forms}.
\\
\\
The {\it inverse problem} we consider here is the following
{\color{black}
\begin{align*}
\text{ \it  Find the support of } (\lambda-\lambda_0,\mu-\mu_0)^T \text{ \it knowing the Neumann-to-Dirichlet operator  } \Lambda(\lambda,\mu).  
 \end{align*}
}
 \noindent
 Next, we take a look at the discrete setting. 
 Let the Neumann boundary $\Gamma_{\textup N}$ be the union of the patches $\Gamma_{\textup N}^{(l)}$, $l=1,...,M$, 
 {\color{black} which are assumed to be relatively open and connected},
 such that 
$\overline{\Gamma_{\textup N}}=\bigcup\limits_{l=1}^M\overline{\Gamma_{\textup N}^{(l)}}$, $\Gamma_{\textup N}^{(i)}\cap \Gamma_{\textup N}^{(j)}=\emptyset$ for $i\neq j$
 and we consider the following problem:
 \begin{align}\label{bvp_1}
\nabla\cdot \left(\lambda (\nabla\cdot u)I + 2\mu \hat{\nabla} u \right)&=0 \quad \,\,\mathrm{in}\,\,\Omega,\\
\left(\lambda (\nabla\cdot u)I + 2\mu \hat{\nabla} u \right) n&=g_l\quad \mathrm{on}\,\, \Gamma_{\textup N}^{(l)},\\
\left(\lambda (\nabla\cdot u)I + 2\mu \hat{\nabla} u \right) n&=0\quad \,\,\mathrm{on}\,\,\Gamma_{\textup N}^{(i)},\quad i\neq l,\\
u&=0 \quad \,\,\mathrm{on}\,\, \Gamma_{\textup D},\label{bvp_4}
\end{align}
 \noindent
 where $g_l$, $l=1,\ldots, M$, denote the $M$ given boundary forces applied to the corresponding
 patches $\Gamma_{\mathrm{N}}^{(l)}$.
\noindent
In order to discretize the {\color{black} Neumann-to-Dirichlet operator}, we apply a boundary force $ g_l$
on the patch $\Gamma_{\mathrm{N}}^{(l)}$ and set
 \begin{align*}
 \Lambda_l^{(k)}(\lambda,\mu):=\int_{\Gamma_{\textup N}^{(l)}}g_l\cdot u^{(k)} \,ds
 \end{align*}
 \noindent
 (cf. (\ref{var-direct_1}) and (\ref{bilinear_Lambda})),
 {\color{black}
  where $u^{(k)}$ solves the corresponding boundary value problem (\ref{bvp_1})-(\ref{bvp_4}) with boundary force $  g_{k}$.
 }  
 \\
 \\
{\color{black}
In Figure \ref{setting_g} a simple example of possible boundary loads $g_l$ and patches $\Gamma_{\mathrm{N}}^{(l)}$ 
is shown.
  \begin{figure}[H]
\begin{center}
  \includegraphics[width=0.35\textwidth]{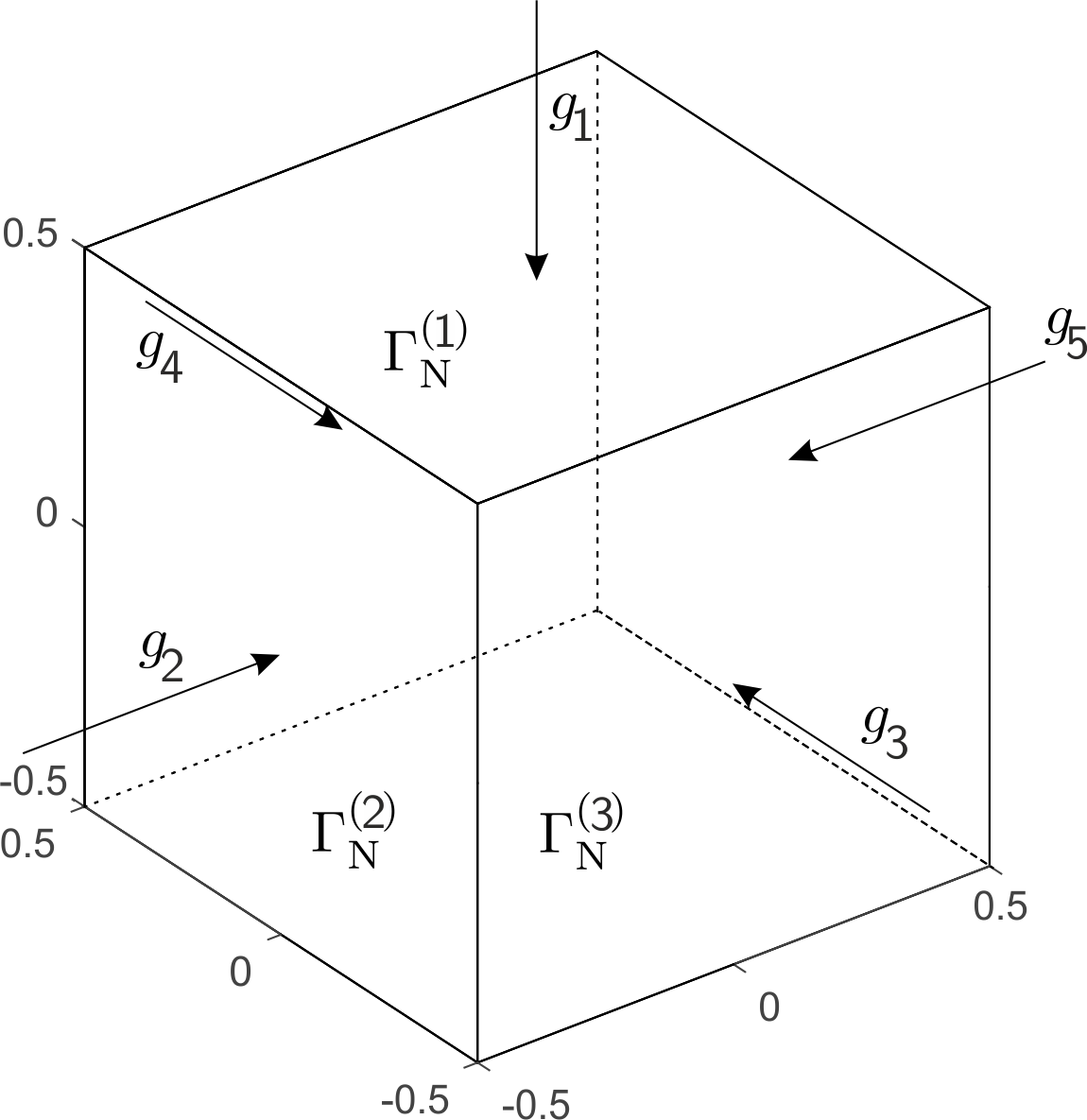}
 \caption{Illustration of possible boundary loads $g_l$ and patches $\Gamma_{\mathrm{N}}^{(l)}$. We consider here $l=1,\ldots,5,$ Neumann patches and one Dirichlet patch (bottom of the cube). The boundary forces $g_l$ are normal vectors from the Euclidean space in each point of the patch.}\label{setting_g}
 \end{center}
  \end{figure}

For the Neumann boundary forces as described here, we get an orthogonal system $g_l$ in $L^2(\Gamma_N)^d$. In practice, we additionally normalize the system $g_l$ and use more patches $\Gamma_{\mathrm{N}}^{(l)}$.
}
 \noindent
 \\
 \\
 For the unknown Lam\'e parameters $(\lambda,\mu)$, we obtain a full matrix 
 \begin{align*}
 \mathbf{\Lambda}(\lambda,\mu)=\left(\Lambda^{(k)}_l(\lambda,\mu)\right)_{k,l=1,...,M}.
 \end{align*}

 \section{Standard One-step Linearization Methods}
 In this section we take a look at {\color{black} one-step} linearization methods. We want to remark that these methods are only a heuristical
 approach but {\color{black} commonly used in practice}.
 \\
 \\
 We compare the matrix of the discretized Neumann-to-Dirichlet operator $\mathbf{\Lambda}(\lambda,\mu)$ with $\mathbf{\Lambda}(\lambda_0,\mu_0)$ for some reference 
 Lam\'e parameter $(\lambda_0,\mu_0)$ in order to reconstruct the 
 difference $(\lambda,\mu)-(\lambda_0,\mu_0)$. Thus, we apply a single linearization step
 \begin{align*}
 \mathbf{\Lambda}^\prime(\lambda_0,\mu_0)\left((\lambda,\mu)-(\lambda_0,\mu_0)\right)\approx
 \mathbf{\Lambda}(\lambda,\mu)-\mathbf{\Lambda}(\lambda_0,\mu_0),
 \end{align*}
 \noindent
 where 
 \begin{align*}
 \mathbf{\Lambda}^\prime(\lambda_0,\mu_0): L^{\infty}(\Omega)^2\to \mathbb{R}^{M\times M}
 \end{align*}
 \noindent
 {\color{black}
 is the Fr\'echet derivative which maps $(\hat{\lambda}, \hat{\mu})\in L^\infty(\Omega)^2$ to
 \begin{align*}
-\left(\int_{\Omega}\hat{\lambda}\left(\nabla\cdot u^{(k)}_{(\lambda_0,\mu_0)}\right) \left(\nabla\cdot u^{(l)}_{(\lambda_0,\mu_0)}\right)
+{\color{black}2}\hat{\mu}\left(\hat{\nabla} u^{(k)}_{(\lambda_0,\mu_0)}\right): \left(\hat{\nabla} u^{(l)}_{(\lambda_0,\mu_0)}\right)dx\right)_{1\leq k, l\leq M} .
 \end{align*}
 }
 \\
   {\color{black} For the solution of the problem,}
  we discretize the reference domain $\overline{\Omega}=\bigcup\limits_{j=1}^{p}\overline{\mathcal{B}}_j$ into $p$ disjoint pixel $\mathcal{B}_j$, 
  where each $\mathcal{B}_j$ is assumed to be open, $\Omega\setminus \mathcal{B}_j$ is connected and 
  $\mathcal{B}_j \cap \mathcal{B}_i=\emptyset$ for $j\neq i$.
  {\color{black} We make a piecewise constant ansatz for $(\kappa,\nu)\approx(\lambda,\mu)-(\lambda_0,\mu_0)$} via
 \begin{align}\label{nu_kappa}
 \kappa(x)=\sum_{j=1}^p \kappa_j \chi_{\mathcal{B}_j}(x)\quad\text{and}\quad 
 \nu(x)=\sum_{j=1}^p \nu_j \chi_{\mathcal{B}_j}(x),
 \end{align} 
 \noindent
 where $\chi_{\mathcal{B}_j}$ is the characteristic function w.r.t. the pixel $\mathcal{B}_j$ and set
 \begin{align*}
   {\boldsymbol \kappa}=(\kappa_j)_{j=1}^p\in \mathbb{R}^p \quad\text{and}\quad
 {\boldsymbol \nu}=(\nu_j)_{j=1}^p\in \mathbb{R}^p.
 \end{align*}
 \noindent
 This approach leads to the linear equation system
 \begin{align}\label{equation_min}
  {\bf S}^{\lambda}{\boldsymbol \kappa}+ {\bf S}^{\mu}{\boldsymbol \nu}= {\bf V},
 \end{align}
 \noindent
 where $\bf V$ and the columns of the sensitivity matrices $\bf S^{\lambda}$ and $\bf S^{\mu}$ contain the entries of 
 {\color{black}$\Lambda(\lambda_0,\mu_0)-\Lambda(\lambda,\mu)$} and the discretized Fr\'echet derivative for a given $\mathcal{B}_j$ for $j=1,...,p$, respectively.
 \noindent
Here, we have
 \begin{align}
 \bf V&=(V_i)_{i=1}^{M^2}\in \mathbb{R}^{M^2}, 
\quad\,{\color{black} V_{(l-1)M+k}=\Lambda_l^{(k)}(\lambda_0,\mu_0)-\Lambda_l^{(k)}(\lambda,\mu)},\\
 \bf S^{\lambda}&=(S^{\lambda}_{ij})\in\mathbb{R}^{M^2,p},
 \quad S^{\lambda}_{(l-1)M+k,j}=\int_{\mathcal{B}_j}\left(\nabla\cdot u^{(k)}_{(\lambda_0,\mu_0)}\right) \left(\nabla\cdot u^{(l)}_{(\lambda_0,\mu_0)}\right)dx,\label{S_lam}\\
 \bf S^{\mu}&=(S^{\mu}_{ij})\in\mathbb{R}^{M^2,p},
 \quad S^{\mu}_{(l-1)M+k,j}=\int_{\mathcal{B}_j}2\left(\hat{\nabla}u_{(\lambda_0,\mu_0)}^{(k)}\right):\left(\hat{\nabla}u_{(\lambda_0,\mu_0)}^{(l)}\right)dx.\label{S_m}
 \end{align}
 \noindent
Solving (\ref{equation_min}) results in a standard minimization problem for the reconstruction of the unknown parameters.
 In order to determine suitable parameters {\color{black}$(\boldsymbol\kappa,\boldsymbol\nu)$},
 we regularize the minimization problem, so that we have
 \begin{align}\label{min_stand}
 \left\Vert \left(\bf S^{\lambda}\,\vert\,\, \bf S^{\mu}\right) 
 \begin{pmatrix}
 \boldsymbol\kappa\\
 \boldsymbol\nu
 \end{pmatrix} 
 - \bf V \right\Vert_{2}^2 + \omega \Vert \boldsymbol \kappa\Vert_{2}^2 + {\color{black}\sigma}\Vert\boldsymbol\nu\Vert_{2}^2\to \mathrm{min!}
 \end{align}
 \noindent
 with $\omega$ and ${\color{black}\sigma}$ as regularization parameters.
For solving this minimization problem we consider the normal equation
 \begin{align*}
 {\bf A}^T {\bf A} 
 \begin{pmatrix}
 \boldsymbol\kappa\\
 \boldsymbol\nu
 \end{pmatrix}
 ={\bf A}^T 
 \begin{pmatrix}
 \bf V\\
 \bf 0\\
 \bf 0
 \end{pmatrix}
 \end{align*}
 \noindent
 with 
$
 \bf A=
 \begin{pmatrix}
 \bf S^{\lambda} & \vert & \bf S^{\mu}\\
 \omega \bf I&\vert & \bf 0 \\
 \bf 0 & \vert & {\color{black}\sigma} \bf I
  \end{pmatrix}.
$
 \noindent
 \\
 \\
Obtaining a solution for this system is memory expensive and finding two suitable parameters $\omega$ and ${\color{black}\sigma}$ can be time consuming,
 since we can only choose them heuristically.
 However, the parameter reconstruction provides good results as shown in the next part.

 \subsection*{Numerical Realization}
 
 We present a simple test model, where we consider a cube of a biological tissue with two inclusions (tumors) as depicted in Figure \ref{standard}.

 \begin{figure}[H]
 \begin{center}
 \includegraphics[width=0.4\textwidth]{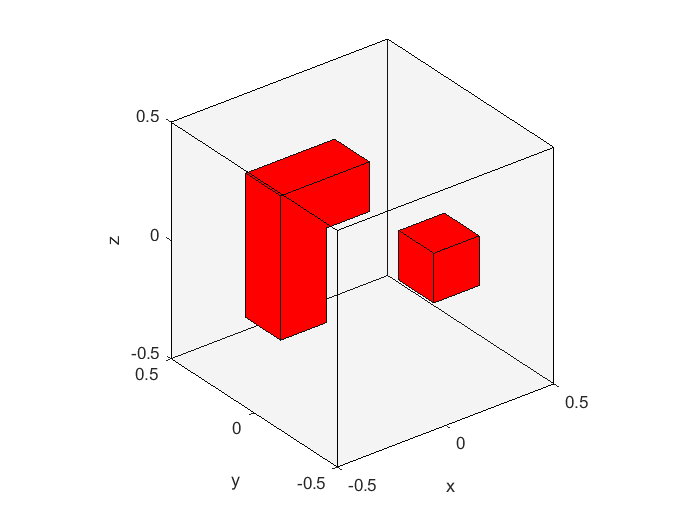}
\caption{Cube with two inclusions (red).}\label{standard}
\end{center}
 \end{figure}
 \noindent
 The Lam\'e parameters of the corresponding materials are given in Table \ref{lame_parameter_mono}.
 \begin{table} [H]
 \begin{center}
 \begin{tabular}{ |c|c| c |}  
\hline
 material & $\lambda_i$ & $\mu_i$ \\
  \hline
$i=0$: tissue &  $6.6211\cdot 10^5$   &  $6.6892\cdot 10^3$   \\
 \hline
$i=1$: tumor &  $2.3177\cdot 10^6$ &  $2.3411\cdot 10^4$  \\
\hline
\end{tabular}
\end{center}
\caption{Lam\'e parameter of the test material in [Pa].}
\label{lame_parameter_mono}
\end{table}
\noindent
For our numerical experiments, we simulate the discrete measurements by solving
\begin{equation}\label{monotonicity_test_1}
\left\{ \begin{array}{rcll}
\nabla\cdot \left(\lambda_0 (\nabla\cdot u_0)I + 2\mu_0 \hat{\nabla} u_0 \right)&=&0&\mathrm{in}\,\,\Omega,\\[1.5ex]
-\nabla\cdot \left(((\lambda_1-\lambda_0)\chi_D)(\nabla\cdot u_0)I+ 2((\mu_1-\mu_0)\chi_D)\hat{\nabla} u_0 \right)\\
+\nabla\cdot \left(\lambda (\nabla\cdot v)I + 2\mu \hat{\nabla} v \right)&=&0 &\mathrm{in}\,\,\Omega,\\ [1.5ex]
\left(\lambda_0 (\nabla\cdot u_0)I + 2\mu_0 \hat{\nabla} u_0 \right) n&=&g_l&\mathrm{on}\,\, \Gamma_{\mathrm{N}},\\ [1.5ex]
\left(\lambda (\nabla\cdot v)I + 2\mu \hat{\nabla} v \right) n&=&0&\mathrm{on}\,\, \Gamma_{\mathrm{N}},\\ [1.5ex]
u_0&=&0 &\mathrm{on}\,\, \Gamma_{\mathrm{D}},\\ [1.5ex]
v&=&0 &\mathrm{on}\,\, \Gamma_{\mathrm{D}},
\end{array}\right.
\end{equation}
\noindent
\\
for each of the $l=1,\ldots,M$, given boundary forces $ g_l$, where $v:=u_0-u$ are the difference measurements. 
The equations regarding $v$ in the system 
(\ref{monotonicity_test_1}) 
result from substracting the boundary value problem (\ref{direct_1}) for the respective Lam\'{e} parameters.

{\color{black}
We want to remark that the Dirichlet boundary is set to the bottom of the cube. The remaining five faces of the cube constitute the Neumann boundary.
Each Neumann face is divided into $25$ squares of equal size ($5\times 5$) resulting in $125$ patches $\Gamma_{\mathrm{N}}^{(l)}$. 
On each $\Gamma_{\mathrm{N}}^{(l)}$, $l=1,\ldots,125,$
we apply a boundary force $g_l$, which is equally distributed on $\Gamma_{\mathrm{N}}^{(l)}$ and pointing in the normal direction of the patch.
}

\subsubsection*{Exact Data}
First of all, we take a look at the example without noise, which means we assume we are given exact data.
 \\
 \\
 In order to obtain a suitable visualization of the $3$D reconstruction, we manipulate the transparency parameter function $\alpha: \mathbb{R} \to [0,1]$
 of Figure \ref{one_step_suitable_3d} as exemplary depicted for the Lam\'e parameter $\mu$ in Figure \ref{trans_color_lin}.
 It should be noted that a low transparency parameter indicates that the corresponding color (here, the colors around zero)
 are plotted with high transparency, while a high $\alpha$ indicates that the corresponding color is plotted opaque.
 The reason for this choice is that values of the calculated difference $\kappa=\mu_1-\mu_0$ close to zero are not an
 indication of an inclusion, while values with a higher absolute value indicate an inclusion. Hence, this choice of
 transparency is suitable to plot the calculated inclusions without being covered by white tetrahedrons with values close
 to zero. Further, the reader should observe that $\alpha(\kappa)>0$ for all values of $\kappa$, so that all tetrahedrons are plotted and that
 the transparency plot for $\nu$ takes the same shape but is adjusted to the range of the calculated values.
\\
 \\
  The following results (Figure \ref{one_step_suitable_3d} - Figure \ref{one_step_suitable_cuts_mu}) are based on a parameter search and
 the regularization parameters are chosen heuristically. 
 Thus, we only present the results with the best parameter choice ($\omega=10^{-17}$ and ${\color{black}\sigma}=10^{-13}$) and 
 reconstruct the difference in the Lam\'{e} parameters 
 $\mu$ and $\lambda$.  
   \begin{figure}[H]
 \begin{center}
  \includegraphics[width=0.35\textwidth]{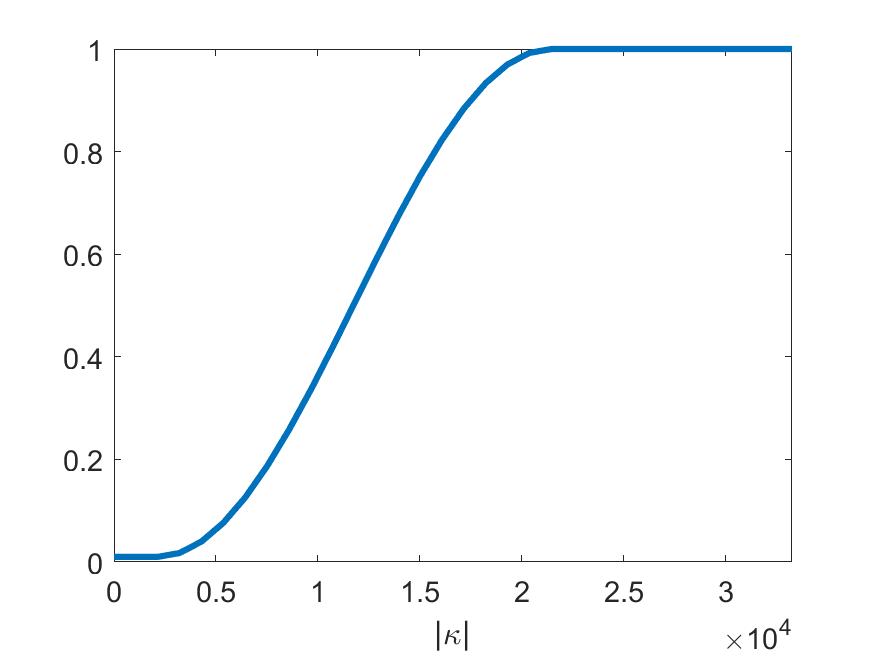}
 \caption{Transparency function for the plots in Figure \ref{one_step_suitable_3d} mapping the values of $\kappa$ to $\alpha(\kappa)$.}\label{trans_color_lin}
 \end{center}
  \end{figure}
  
  \begin{figure}[H]
\begin{center}
  \includegraphics[width=0.85\textwidth]{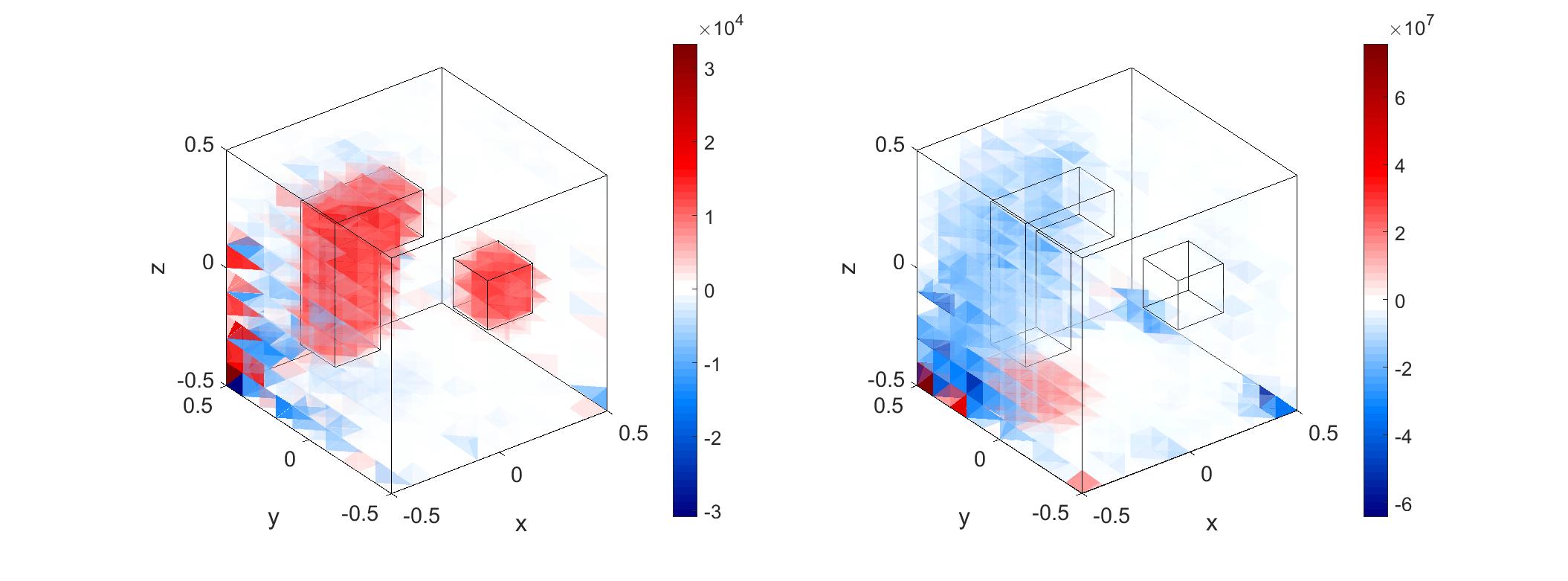}
 \caption{Shape reconstruction of two inclusions of the difference in the Lam\'{e} parameter $\mu$ (left hand side) and $\lambda$ (right hand side) for the regularization parameters
 \mbox{$\omega=10^{-17}$} and ${\color{black}\sigma}=10^{-13}$ without noise and transparency function $\alpha$ as shown in
 Figure \ref{trans_color_lin}.}\label{one_step_suitable_3d}
 \end{center}
  \end{figure}

  \begin{figure}[H]
  \begin{center}
  \includegraphics[width=0.32\textwidth]{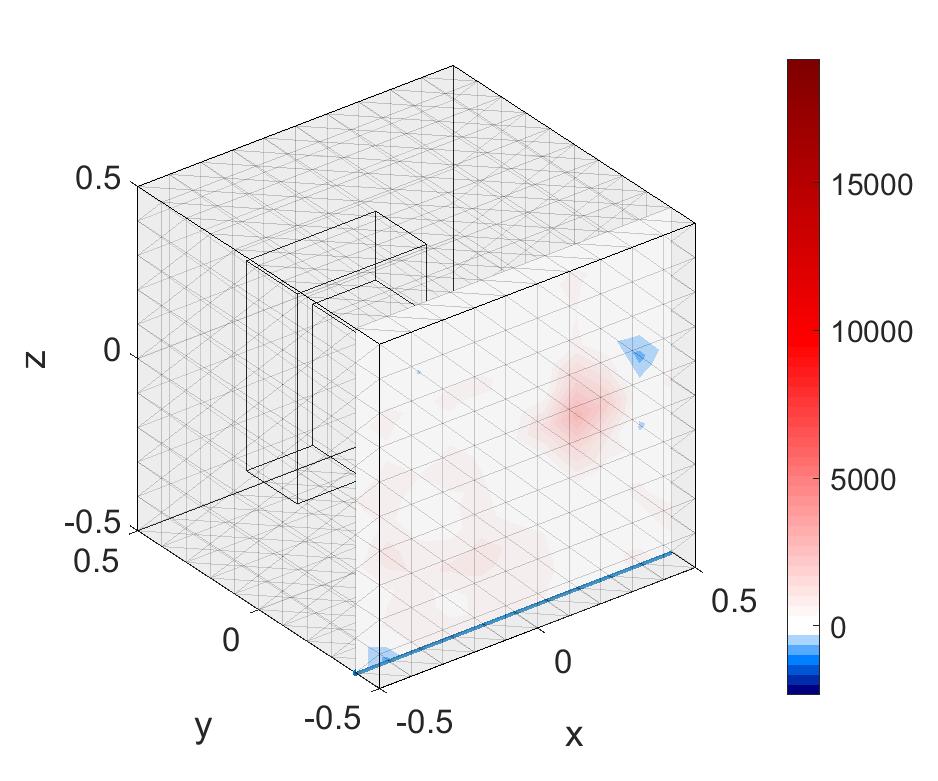}
  \includegraphics[width=0.32\textwidth]{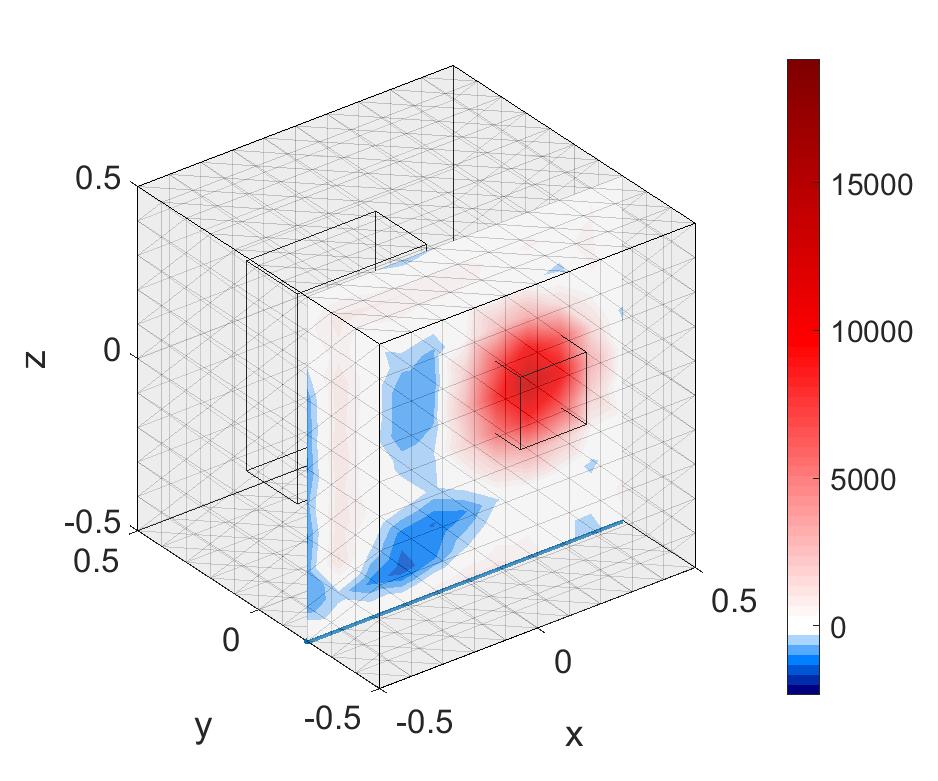}\\
   \includegraphics[width=0.32\textwidth]{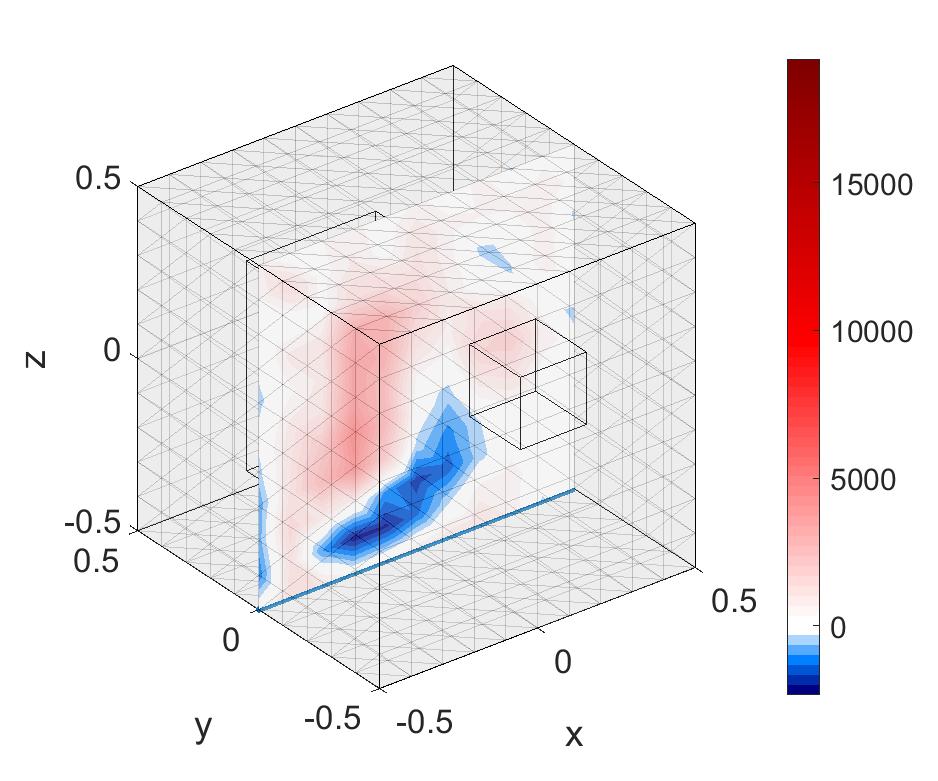}
   \includegraphics[width=0.32\textwidth]{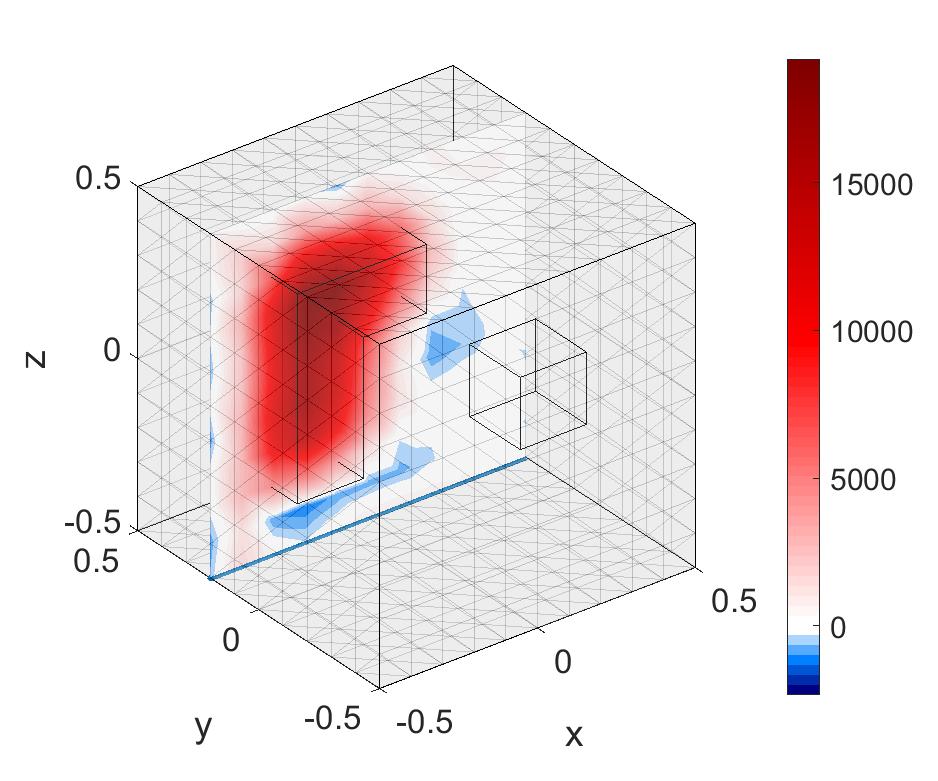}
  \includegraphics[width=0.32\textwidth]{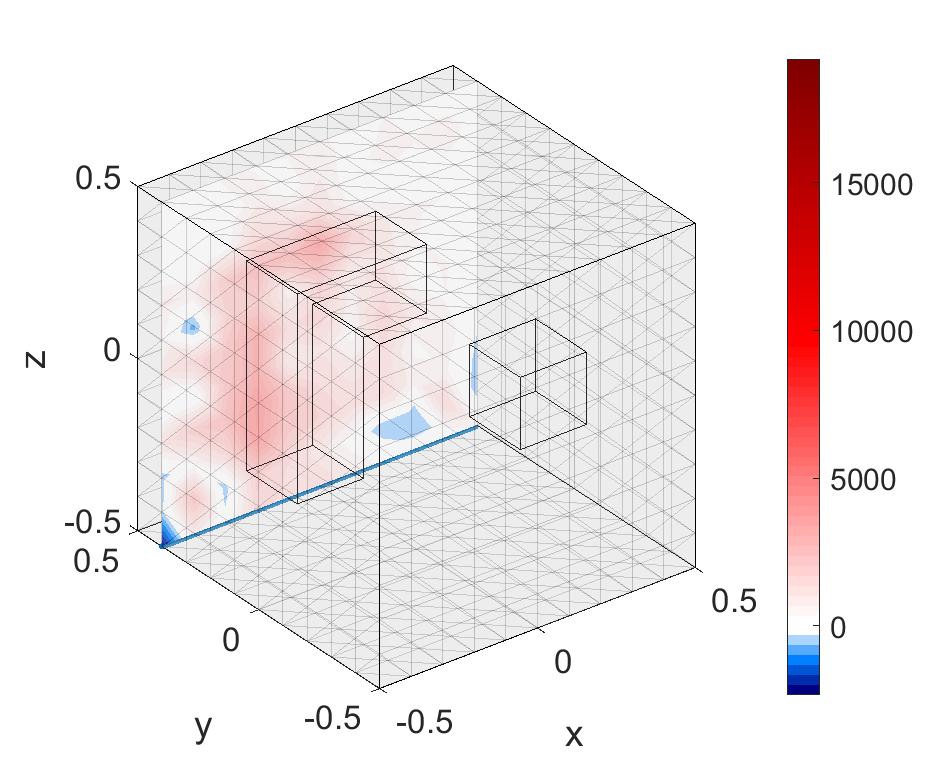}
 \caption{Shape reconstruction of two inclusions of the reconstructed difference in the Lam\'{e} parameter $\mu$ for the regularization parameters
 $\omega=10^{-17}$ and ${\color{black}\sigma}=10^{-13}$ depicted as cuts without noise.}\label{one_step_suitable_cuts_mu}
 \end{center}
  \end{figure}

\noindent
 With these regularization parameters, the two inclusions are detected and reconstructed correctly for $\mu$ 
 (see Figure \ref{one_step_suitable_3d} in the left hand side)
 and the value of $\mu-\mu_0$ is in the correct amplitude range as depicted in Figure \ref{one_step_suitable_cuts_mu}.
Figure \ref{one_step_suitable_3d} shows us, that for $\lambda-\lambda_0$, the reconstruction 
does not work. 
The reason is that the range of the  Lam\'{e} parameters differs from each other around $10^2$ Pa ($\lambda\approx100\cdot\mu$), but
\begin{align*}
\Vert {\bf S}^\mu\Vert_2\approx 1.2\cdot 10^4 \Vert {\bf S}^\lambda\Vert_2, 
\end{align*}
\noindent
i.e. the signatures of $\mu$
are represented far stronger in the calculation of $\bf V$ than those of $\lambda$.

\subsubsection*{Noisy Data}

Next, we go over to noisy data. We assume that we are given a noise level $\eta\geq 0$ and set

\begin{align}\label{def_delta}
\delta=\eta \cdot ||\boldsymbol{\Lambda}(\lambda,\mu)||_F.
\end{align}
\noindent
Further, we define $\boldsymbol{\Lambda}^\delta(\lambda,\mu)$ as
\begin{align}\label{def_Lambda_noise}
\boldsymbol{\Lambda}^\delta(\lambda,\mu)=\boldsymbol{\Lambda}(\lambda,\mu)+\delta \overline{{\bf E}},
\end{align}
\noindent
with $\overline{{\bf E}}={\bf E}/||{\bf E}||_F$, where ${\bf E}$ consists of $M\times M$ random 
{\color{black} uniformly distributed}
values in $[-1,1]$.
\noindent
{\color{black}
We set 
\begin{align*}
{\bf  V}^\delta=\boldsymbol{\Lambda}(\lambda_0,\mu_0)-\boldsymbol{\Lambda}^\delta(\lambda,\mu).
\end{align*}
}
\noindent
Hence, we have 
\begin{align*}
 \Vert{\bf  V}^\delta - {\bf V}\Vert \leq \delta.
\end{align*}
\noindent
In the following examples, we consider relative noise levels of $\eta=1\%$  (Figure \ref{result_one_step_noise} ) and 
$\eta=10\%$ (Figure \ref{one_step_suitable_3d_noise} and \ref{one_step_suitable_cuts_mu_noise}) with respect to
the Frobenius norm as given in (\ref{def_Lambda_noise}), where
the regularization parameters are chosen heuristically and given in the caption of the figure.
\noindent
\\
\noindent
In Figure \ref{result_one_step_noise}, we observe that for a low noise level with $\eta=1\%$, we obtain
a suitable reconstruction of the inclusion concerning the Lam\'e parameter $\mu$ and the reconstruction 
of $\lambda$ fails again.
\\
\\
    \begin{figure}[H]
  \begin{center}
  \includegraphics[width=0.85\textwidth]{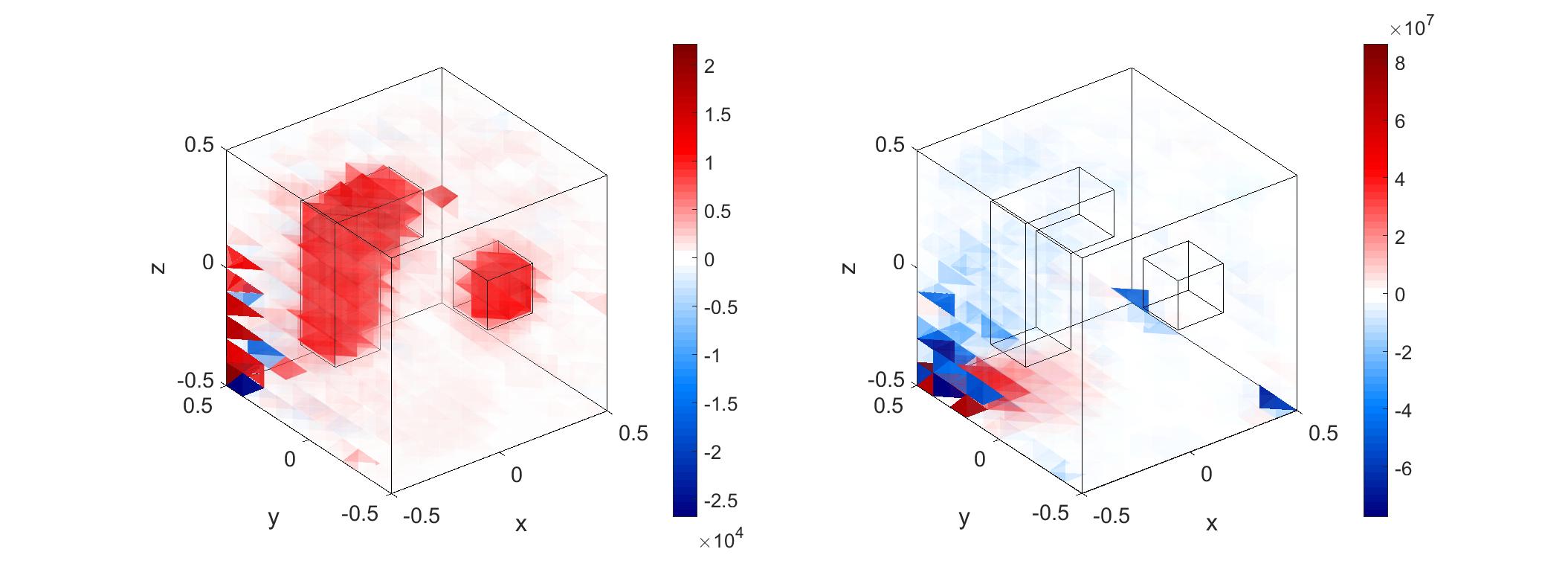}
 \caption{Shape reconstruction of two inclusions of the difference in the Lam\'{e} parameter $\mu$ (left hand side) and $\lambda$ 
 (right hand side) for the regularization parameters
 $\omega=1.1\cdot 10^{-17}$ and ${\color{black}\sigma}=1.1\cdot 10^{-13}$ with relative noise $\eta=1\%$
 and transparency function $\alpha$ as shown in Figure \ref{trans_color_lin}.}\label{result_one_step_noise}
\end{center}
  \end{figure}
  
  \begin{figure}[H]
  \begin{center}
  \includegraphics[width=0.32\textwidth]{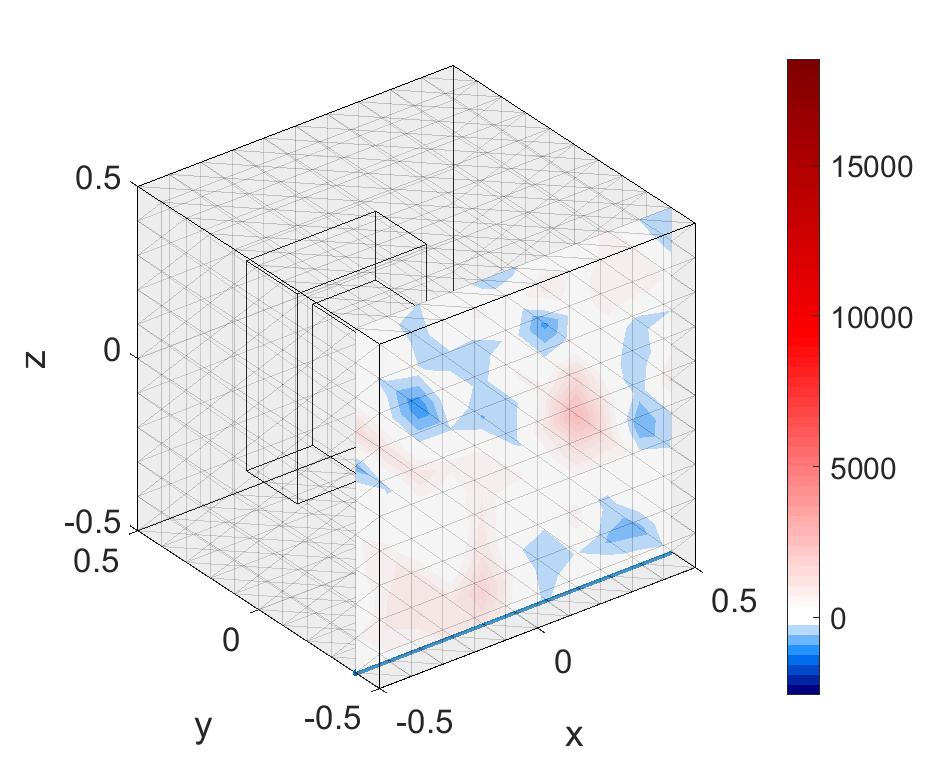}
  \includegraphics[width=0.32\textwidth]{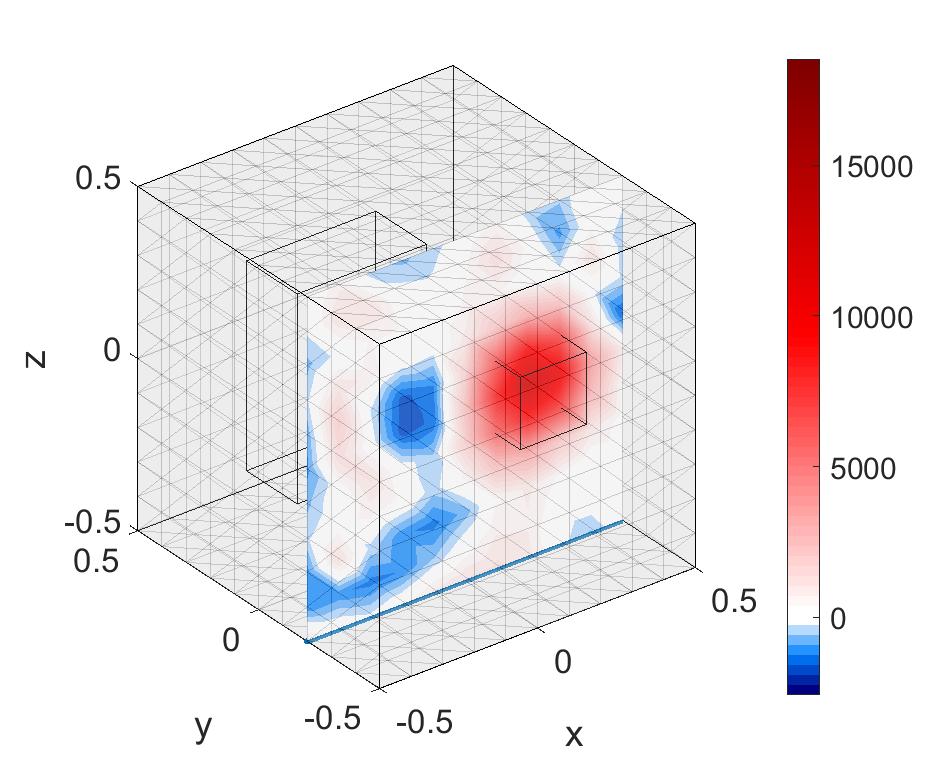}\\
  \includegraphics[width=0.32\textwidth]{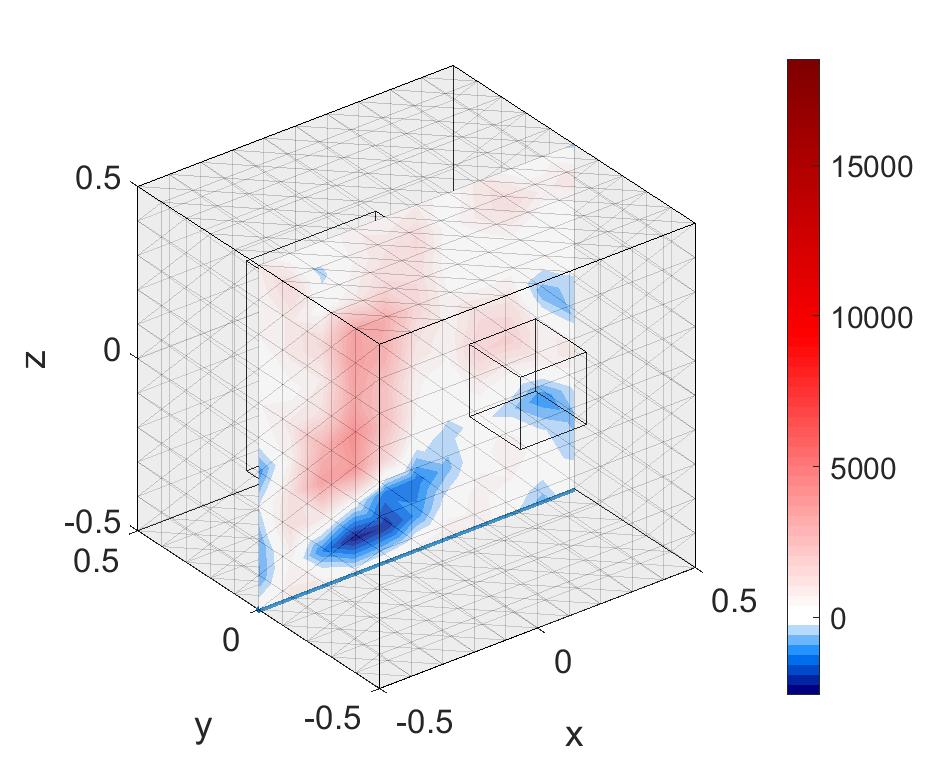}
\includegraphics[width=0.32\textwidth]{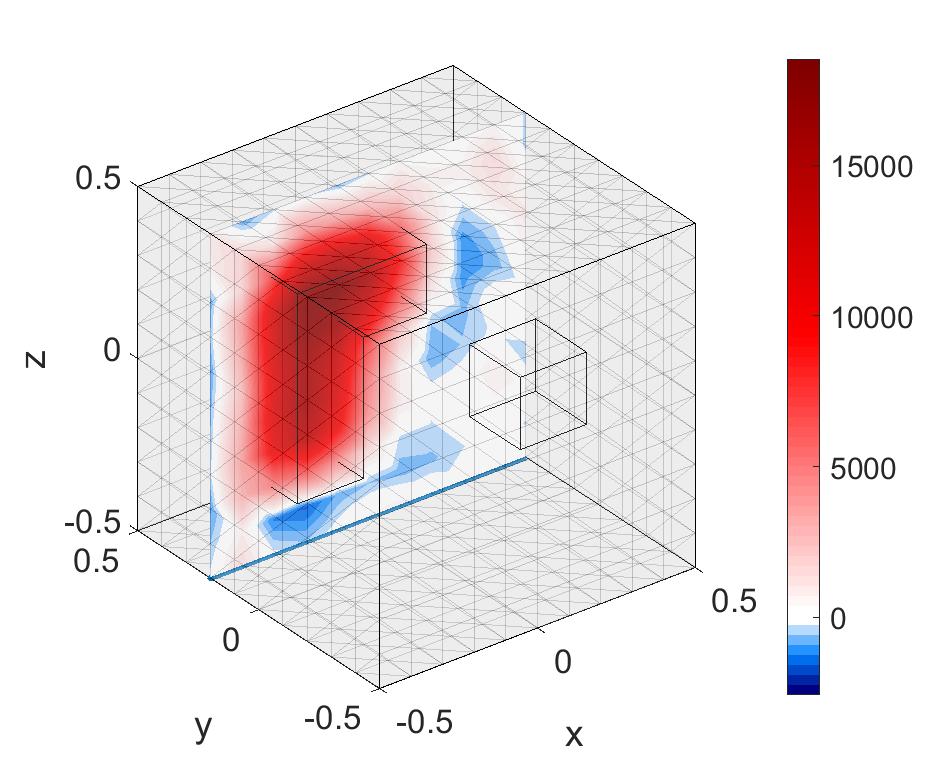}
  \includegraphics[width=0.32\textwidth]{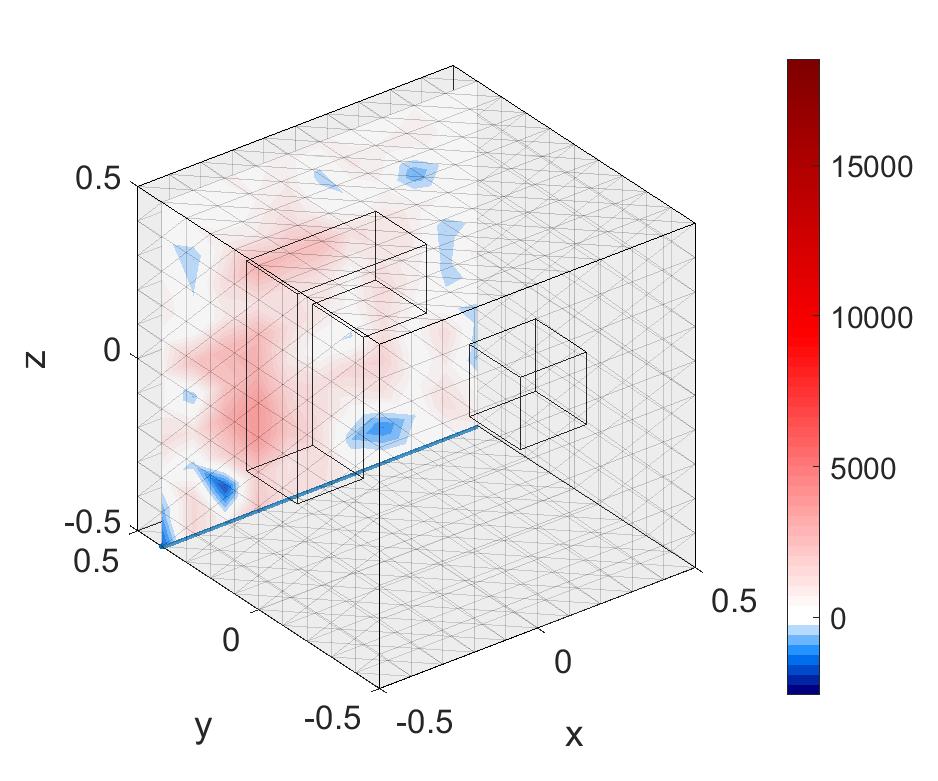}
 \caption{Shape reconstruction of two inclusions of the reconstructed difference in the Lam\'{e} parameter $\mu$ for the regularization parameters
$\omega=1.1\cdot 10^{-17}$ and ${\color{black}\sigma}=1.1\cdot 10^{-13}$ depicted as cuts with relative noise $\eta=1\%$.}
 \label{one_step_suitable_cuts_mu_noise_1}
 \end{center}
  \end{figure}
  
\noindent
In contrary to the low noise level ($\eta=1\%$), Figures \ref{one_step_suitable_3d_noise} and \ref{one_step_suitable_cuts_mu_noise} show us that the standard one-step
linearization method has problems in handling higher noise levels ($\eta=10\%$). As such, in the $3$D reconstruction 
(see Figure \ref{one_step_suitable_3d_noise}) it is hard to recognize the two inclusions even with respect to
the Lam\'e parameter $\mu$. Further on in the plots of the cuts in Figure \ref{one_step_suitable_cuts_mu_noise},
 the reconstructions of the inclusions are blurred out.
 \\
 \\
  \begin{figure}[H]
  \begin{center}
  \includegraphics[width=0.85\textwidth]{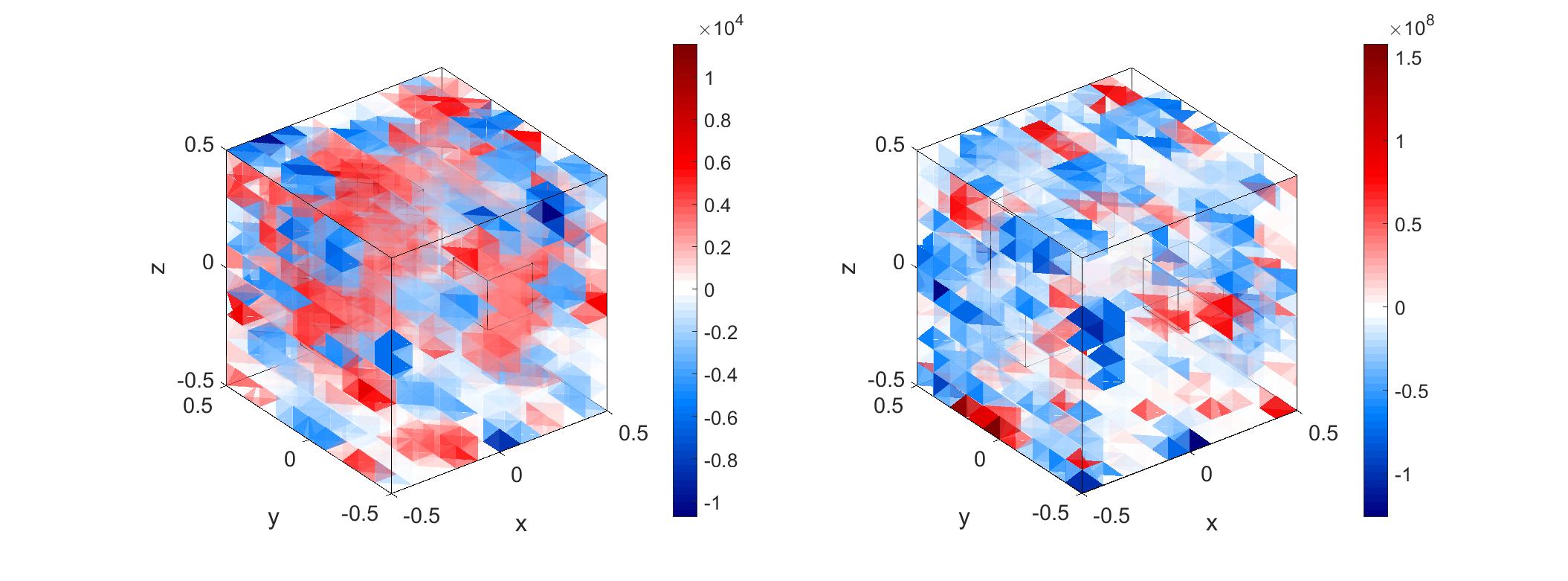}
 \caption{Shape reconstruction of two inclusions of the difference in the Lam\'{e} parameter $\mu$ (left hand side) and $\lambda$ (right hand side) for the regularization parameters
 $\omega=6\cdot 10^{-17}$ and ${\color{black}\sigma}=6\cdot 10^{-13}$ with relative noise $\eta=10\%$
 and transparency function $\alpha$ as shown in Figure \ref{trans_color_lin}}\label{one_step_suitable_3d_noise}
 \end{center}
  \end{figure}
  
  \begin{figure}[H]
  \begin{center}
  \includegraphics[width=0.32\textwidth]{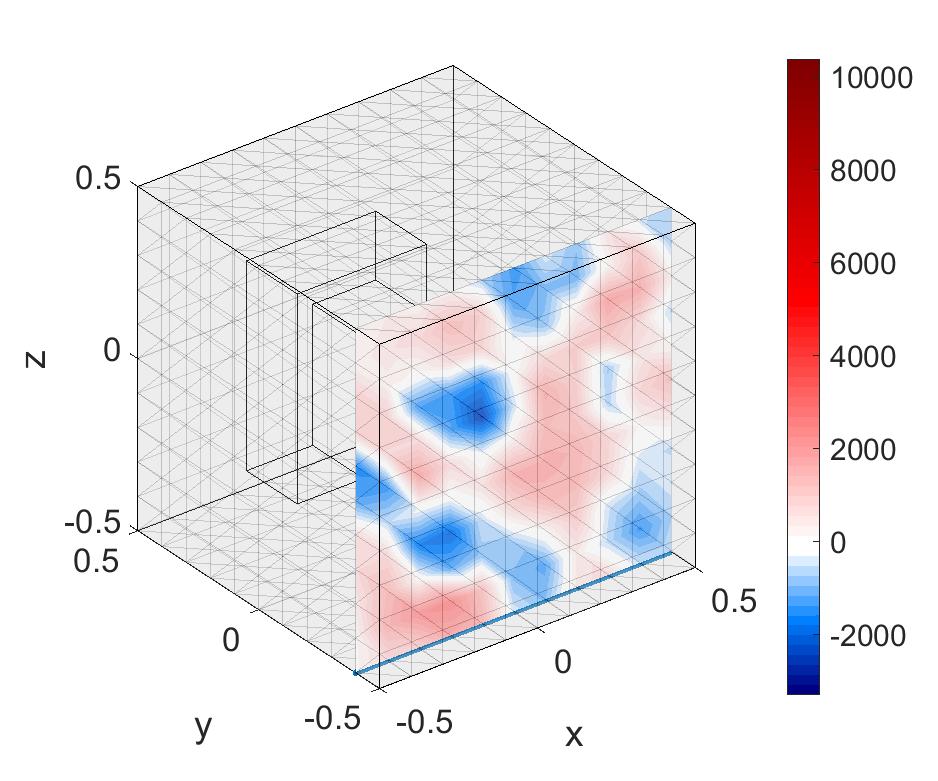}
  \includegraphics[width=0.32\textwidth]{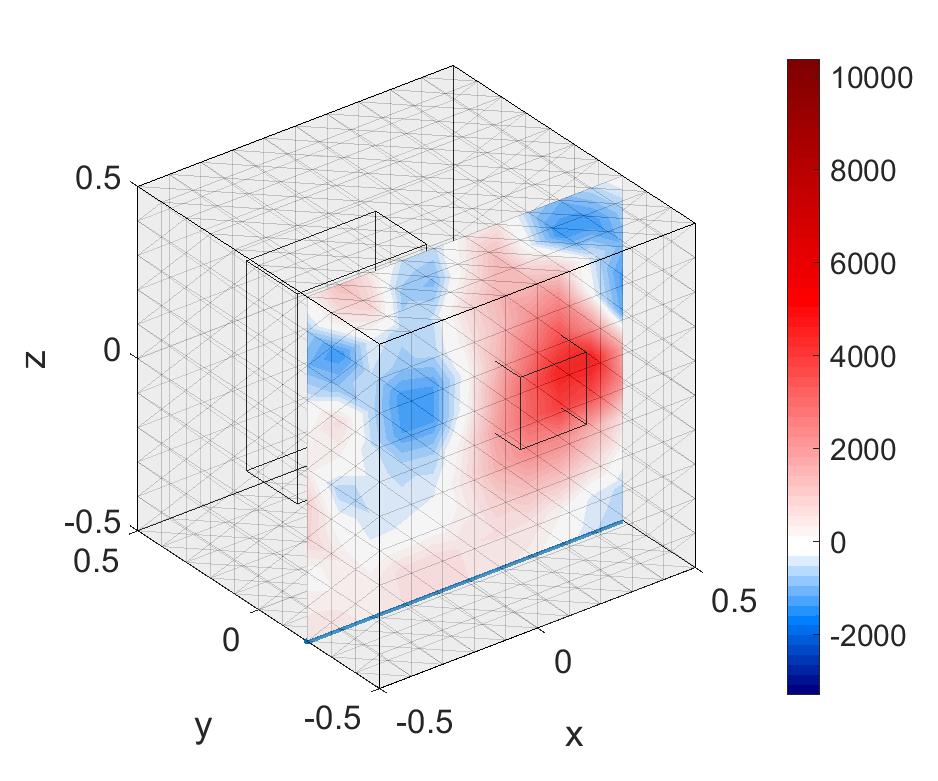}\\
  \includegraphics[width=0.32\textwidth]{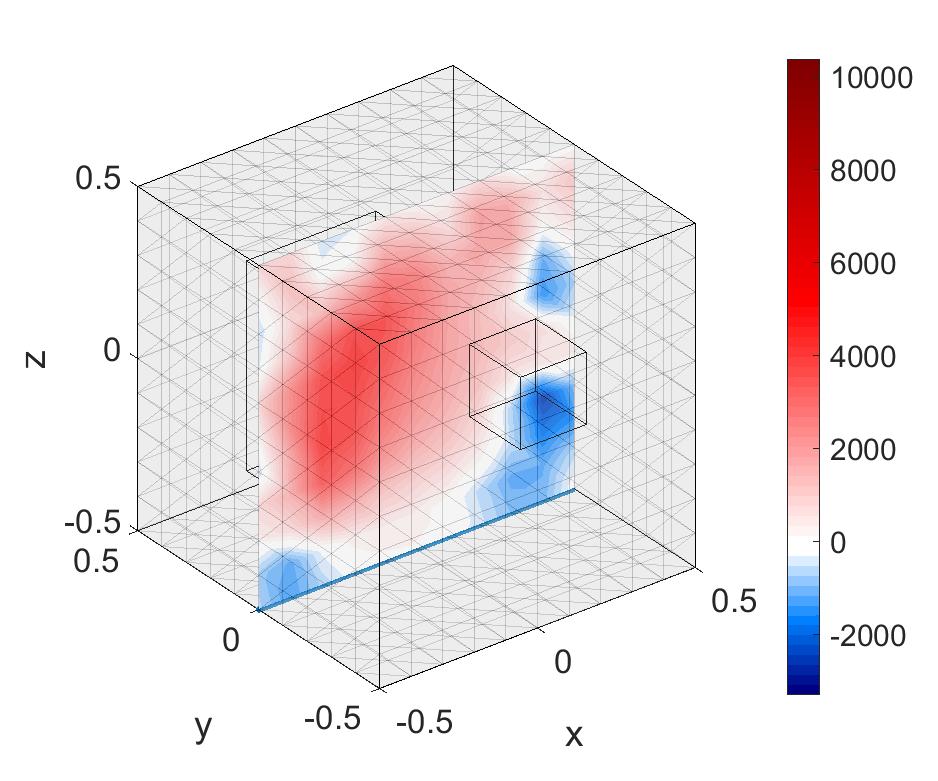}
\includegraphics[width=0.32\textwidth]{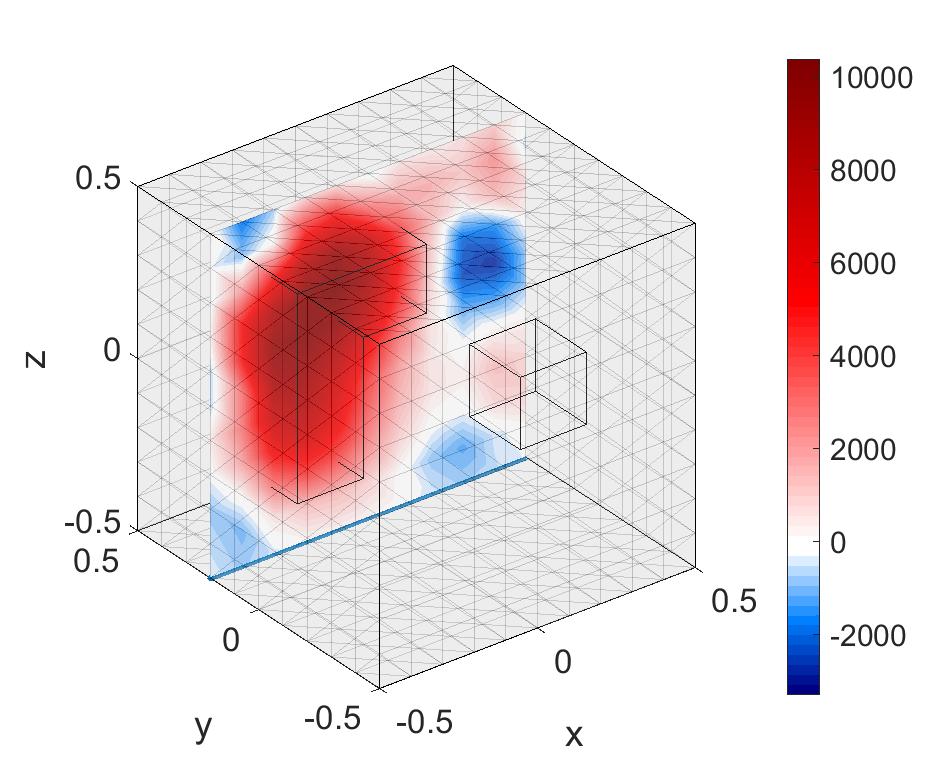}
  \includegraphics[width=0.32\textwidth]{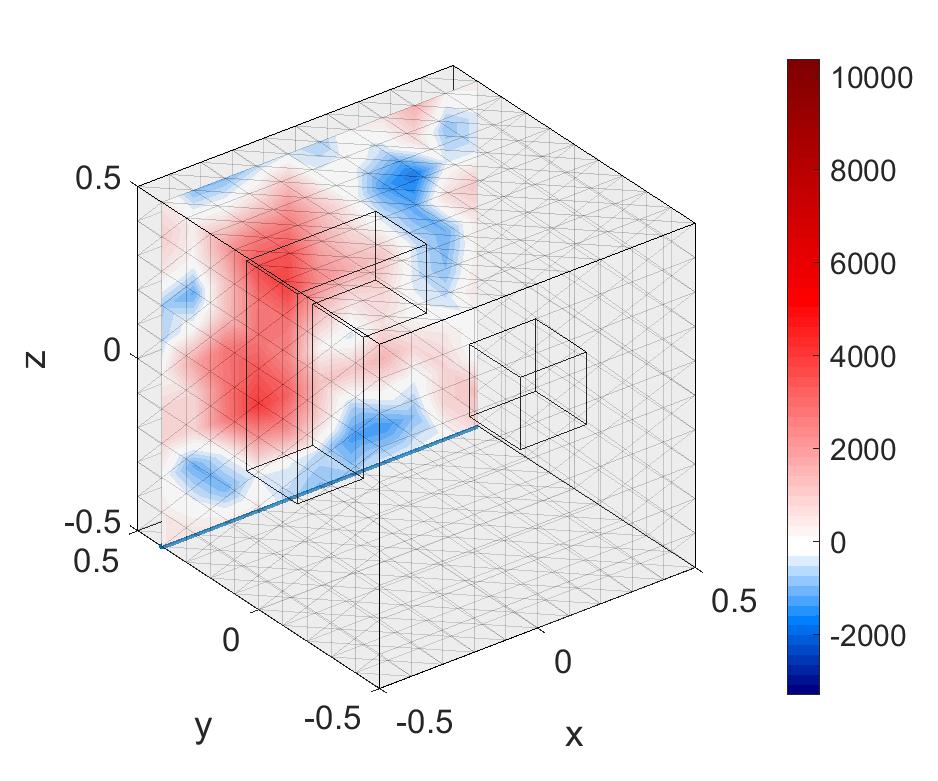}
 \caption{Shape reconstruction of two inclusions of the reconstructed difference in the Lam\'{e} parameter $\mu$ for the regularization parameters
 $\omega=6 \cdot 10^{-17}$ and ${\color{black}\sigma}=6\cdot 10^{-13}$ depicted as cuts with relative noise $\eta=10\%$.}
 \label{one_step_suitable_cuts_mu_noise}
 \end{center}
  \end{figure}
  \noindent
 \begin{remark}
  All in all, the numerical experiments of this section motivate the consideration of a modified minimization problem in order
 to obtain a stable method for noisy data as well as a good reconstruction for the Lam\'{e} parameter $\lambda$. 
 In doing so, we will combine the idea of the standard 
 one-step linearization with the monotonicity method.
 \end{remark}

 \section{Enhancing the Standard Residual-based Minimization Problem}

We summarize and present the required results concerning the monotonicity properties of the Neumann-to-Dirichlet operator
as well as the monotonicity methods introduced and proven in \cite{Eberle_Monotonicity} and \cite{Eberle_mon_test}.
 
 \subsection{Summary of the Monotonicity Methods}

First, we state the monotonicity estimates for the Neumann-to-Dirichlet operator $\Lambda(\lambda,\mu)$
{\color{black}
and denote by $u^g_{(\lambda,\mu)}$ the solution of problem (\ref{direct_1})-(\ref{direct_3}) for the boundary load
$g$ and the Lam\'e parameters $\lambda$ and $\mu$.
}
\begin{lemma}[Lemma 3.1 from \cite{Eberle_Monotonicity}]
\label{mono}
Let {\color{black}$(\lambda_1,\mu_1),(\lambda_2,\mu_2)  \in  L_+^\infty(\Omega)\times L_+^\infty(\Omega) $},  $g\in L^2(\Gamma_{\textup N})^d$ be an applied boundary force, and let 
$u_1:=u^{g}_{(\lambda_1,\mu_1)}\in \mathcal{V}$, $u_2:=u^{g}_{(\lambda_2,\mu_2)}\in \mathcal{V}$. Then
\begin{align}
\label{eqmono}
\int_\Omega &2(\mu_1-\mu_2)\hat{\nabla}u_2:\hat{\nabla}u_2+(\lambda_1-\lambda_2)\nabla\cdot u_2\nabla\cdot u_2\,dx\\ \nonumber
&\geq \langle g,\Lambda(\lambda_2,\mu_2)g\rangle-\langle g,\Lambda(\lambda_1,\mu_1)g\rangle\\ 
&\geq \int_\Omega 2(\mu_1-\mu_2)\hat{\nabla}u_1 :\hat{\nabla}u_1+(\lambda_1-\lambda_2)\nabla\cdot u_1 \nabla\cdot u_1\,dx.
\end{align}
\end{lemma}

\noindent

\begin{lemma}[Lemma 2 from \cite{Eberle_mon_test}]\label{mono_2}
Let {\color{black}$(\lambda_1,\mu_1),(\lambda_2,\mu_2)  \in  L_+^\infty(\Omega)\times L_+^\infty(\Omega) $},   $g\in L^2(\Gamma_{\textup N})^d$ be an applied boundary force, and let 
\mbox{$u_1:=u^{g}_{(\lambda_1,\mu_1)}\in \mathcal{V}$}, $u_2:=u^{g}_{(\lambda_2,\mu_2)}\in \mathcal{V}$. Then
\begin{align}\label{mono_lame}
\langle g&,\Lambda(\lambda_2,\mu_2)g\rangle-\langle g,\Lambda(\lambda_1,\mu_1)g\rangle\\  \nonumber
&\geq 
\int_{\Omega} 2\left(\mu_2-\frac{\mu_2^2}{\mu_1}\right)\hat{\nabla}u_2:\hat{\nabla}u_2\,dx
+\int_{\Omega}\left(\lambda_2-\frac{\lambda_2^2}{\lambda_1}\right)\nabla\cdot u_2 \nabla \cdot u_2\,dx\\
&=
\int_{\Omega} 2\frac{\mu_2}{\mu_1}\left(\mu_1-\mu_2\right)\hat{\nabla}u_2:\hat{\nabla}u_2\,dx
+\int_{\Omega}\frac{\lambda_2}{\lambda_1}\left(\lambda_1-\lambda_2\right) \nabla\cdot u_2 \nabla \cdot u_2\,dx.
\end{align}
\end{lemma}

\noindent
{\color{black}
As in the previous section, we denote by $(\lambda_0,\mu_0)$ the material without inclusion. Following Lemma \ref{mono},
we have
}

\begin{corollary}[Corollary 3.2 from \cite{Eberle_Monotonicity}]\label{monotonicity}
For $(\lambda_0,\mu_0),(\lambda_1,\mu_1) \in L_+^\infty(\Omega)\times L_+^\infty(\Omega) $
\begin{align}
\lambda_0\leq \lambda_1 \text{ and } \mu_0\leq \mu_1 \quad \text{  implies } \quad \Lambda(\lambda_0,\mu_0)\geq \Lambda(\lambda_1,\mu_1).
\end{align}
\end{corollary}
\noindent
Further on, we give a short overview concerning the monotonicity methods, where  we restrict ourselves to the case $\lambda_1\geq \lambda_0$, $\mu_1\geq \mu_0$.
In the following, let $\mathcal{D}$ be the unknown inclusion and $\chi_\mathcal{D}$ the characteristic function w.r.t. $\mathcal{D}$. {\color{black} In addition, we deal with "noisy difference measurements",
i.e. distance measurements between $u^g_{(\lambda,\mu)}$ and $u^g_{(\lambda_0,\mu_0)}$ affected by noise,
which stem from system (\ref{monotonicity_test_1}).}
\\
\\
{\color{black}
We define the outer support in correspondence to \cite{Eberle_mon_test} as follows: 
let $\phi=(\phi_1,\phi_2):\Omega\to \mathbb{R}^2$ be a measurable
function,
the outer support $\underset{\partial\Omega}{\mathrm{out}}\,\mathrm{supp}(\phi)$ is the 
 complement (in $\overline{\Omega}$) of the union of those relatively open $U\subseteq\overline{\Omega}$
 that are connected to $\partial\Omega$ and for which $\phi\vert_{U}=0$,
\noindent
respectively.
}
\\
\begin{corollary}{Linearized monotonicity test }(Corollary 2.7 from \cite{Eberle_mon_test}\label{lin_mon_1})
\\
Let $\lambda_0$, $\lambda_1$, $\mu_0$, $\mu_1\in\mathbb{R}^+$ with $\lambda_1>\lambda_0$, $\mu_1>\mu_0$  
and assume that 
$(\lambda,\mu)=(\lambda_0+(\lambda_1-\lambda_0)\chi_\mathcal{D},\mu_0+(\mu_1-\mu_0)\chi_{\mathcal{D}})$
with $\mathcal{D}=\mathrm{out}_{\partial\Omega}\,\mathrm{supp}((\lambda-\lambda_0,\mu-\mu_0)^T)$.
Further on let $\alpha^\lambda,\alpha^\mu\geq 0$, $\alpha^\lambda+\alpha^\mu>0$ and $\alpha^\lambda \leq \tfrac{\lambda_0}{\lambda_1}(\lambda_1 -\lambda_0)$,  $\alpha^\mu \leq\tfrac{\mu_0}{\mu_1} (\mu_1 -\mu_ 0)$.
Then for every open set $\mathcal{B}$
\begin{align*}
\mathcal{B}\subseteq\mathcal{D}\quad\text{if and only if}\quad \Lambda(\lambda_0,\mu_0)+\Lambda^\prime(\lambda_0,\mu_0)(\alpha^\lambda \chi_\mathcal{B}, \alpha^\mu \chi_\mathcal{B})\geq \Lambda(\lambda,\mu).
\end{align*}
\end{corollary}
\noindent

\begin{corollary}{Linearized monotonicity test for noisy data }(Corollary 2.9 from \cite{Eberle_mon_test}\label{lin_mono_test_noise})
\\
Let $\lambda_0$, $\lambda_1$, $\mu_0$, $\mu_1\in\mathbb{R}^+$ with $\lambda_1>\lambda_0$, $\mu_1>\mu_0$  
and assume that 
$(\lambda,\mu)=(\lambda_0+(\lambda_1-\lambda_0)\chi_\mathcal{D},\mu_0+(\mu_1-\mu_0)\chi_{\mathcal{D}})$
with $\mathcal{D}=\mathrm{out}_{\partial\Omega}\,\mathrm{supp}((\lambda-\lambda_0,\mu-\mu_0)^T)$.
Further on, let $\alpha^\lambda,\alpha^\mu\geq 0$, $\alpha^\lambda+\alpha^\mu>0$ with $\alpha^\lambda \leq \frac{\lambda_0}{\lambda_1} (\lambda_1 -\lambda_0) $,
$\alpha^\mu\leq \frac{\mu_0}{\mu_1} (\mu_1 -\mu_0) $.
Let $\Lambda^\delta$ be the Neumann-to-Dirichlet operator for noisy difference measurements with noise level $\delta>0$.
Then for every open set $\mathcal{B}\subseteq\Omega$ there exists a noise level $\delta_0>0$, such that for all
$0<\delta<\delta_0$, $\mathcal{B}$ is correctly detected as inside or not inside the inclusion $\mathcal{D}$ 
by the following monotonicity test
\begin{eqnarray*}
\mathcal{B}\subseteq\mathcal{D}\quad \textnormal{\it if and only if} \quad
\Lambda(\lambda_0,\mu_0) +\Lambda^\prime(\lambda_0,\mu_0)(\alpha^\lambda\chi_{\mathcal{B}},\alpha^\mu\chi_{\mathcal{B}})-\Lambda^\delta(\lambda,\mu)+\delta I \geq 0.
\end{eqnarray*}
\end{corollary}
 \noindent
 Finally, we present the result (see Figure \ref{result_lin_noise}) obtained from noisy data $\Lambda^{\delta}$ 
 with the linearized monotonicity method as described in Corollary \ref{lin_mono_test_noise}, where we 
 use the same pixel partition as for the one-step linearization method.

   \begin{figure}[H]
  \begin{center}
  \includegraphics[width=0.4\textwidth]{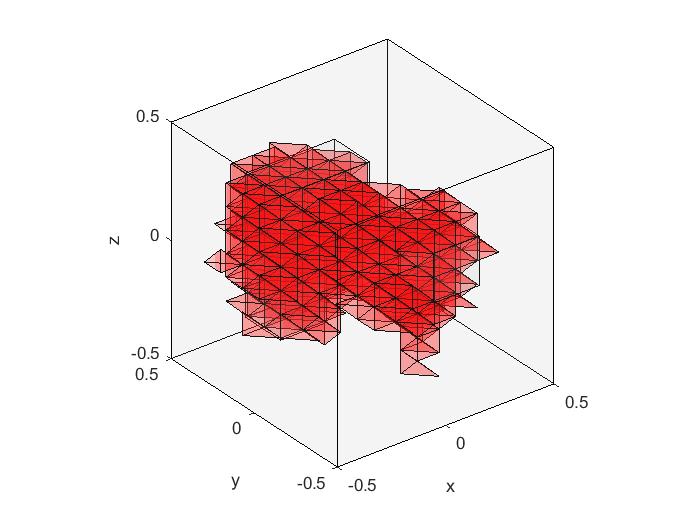}
 \caption{Shape reconstruction of two inclusions (red) for
$\alpha^\lambda=0.28(\lambda_1-\lambda_0)\approx 4.6\cdot 10^5$Pa, $\alpha^\mu=0.28(\mu_1-\mu_0)\approx 4.7\cdot 10^3$ Pa 
with relative noise $\eta=0.1\%$ and $\delta=1.88\cdot 10^{-10}$.}\label{result_lin_noise}
\end{center}
  \end{figure}
 
  \begin{remark}
 The linearized monotonicity method converges theoretically rigorously, but in practice delivers poorer reconstructions
 even for small noise (see Figure \ref{result_lin_noise}, where the two inclusions are not separated)
 than the theoretically unproven heuristic one-step linearization (see Figure \ref{result_one_step_noise},
 where the two inclusions are separated).
 Thus, we improve the standard one-step 
 linearization method by combining it with the monotonicity method without losing the convergence results.
 \end{remark}

 \subsection{Monotonicity-based Regularization}
 \hspace{-0.2cm} We assume again
 that the background $(\lambda_0,\mu_0)$ is homogeneous and that the contrasts of the anomalies
 \mbox{$(\gamma^{\lambda},\gamma^{\mu})^T\in L^\infty_+(\mathcal{D})^2$} with
 \begin{align*}
 \begin{pmatrix}
  \lambda (x)\\
  \mu(x)
 \end{pmatrix}
 =
 \begin{pmatrix}
  \lambda_0+\gamma^{\lambda}(x)\chi_\mathcal{D}(x)\\
  \mu_0+\gamma^{\mu}(x)\chi_\mathcal{D}(x)
 \end{pmatrix},
 \end{align*}
 \noindent
are bounded for all $x\in\mathcal{D}$ (a.e.) via
 \begin{align*}
 c^\lambda \leq \gamma^\lambda(x)\leq C^\lambda\quad\text{and}\quad
 c^\mu\leq \gamma^\mu(x)\leq C^\mu,
 \end{align*}
 \noindent 
 with $c^\lambda$, $C^\lambda$, $c^\mu$, $C^\mu\geq 0$.
 $\mathcal{D}$ is an open set denoting the anomalies and the parameters $\lambda_0,\mu_0, c^\lambda,c^\mu,C^\lambda$ and $C^\mu$ are assumed to be known. {\color{black} In addition, we want to remark that $\Omega\setminus \mathcal{D}$ has to be connected.}
{\color{black}
In doing so, we can also handle more general Lam\'e parameters and not only piecewise constant parameters as in 
the previous section.
}
 \\
\\
 {\color{black}
Here, we focus
 on the case $\lambda\geq \lambda_0$, $\mu\geq \mu_0$, while the case $\lambda\leq \lambda_0$, $\mu\leq \mu_0$
 can be found in the Appendix.
 \\
 \\
 Similar as in the one-step linearization method, we make the piecewise constant ansatz (\ref{nu_kappa}) in order 
 to approximate $(\gamma^{\lambda},\gamma^{\mu})$ by $(\kappa,\nu).$
 \\
 \\
 The main idea of monotonicity-based regularization is to minimize the residual of the linearized problem, i.e.,
  \begin{align}\label{min_prob_mon_based}
 \left\Vert \left(\bf S^{\lambda}\,\vert\,\, \bf S^{\mu}\right) 
 \begin{pmatrix}
 \boldsymbol\kappa\\
 \boldsymbol\nu
 \end{pmatrix} 
 - \bf V\right\Vert^2_{2} \to \mathrm{min!}
 \end{align}
 with constraints on $(\boldsymbol{\kappa},\boldsymbol{\nu})$ that are obtained from the monotonicity
 properties introduced in Lemma \ref{mono} and \ref{mono_2}.
Our aim is to rewrite the minimization problem (\ref{min_prob_mon_based}) for the case 
$\mu_0\neq \mu,\lambda_0\neq \lambda$ in $\mathcal{D}$
in order to be able to reconstruct the inclusions
also with respect to $\lambda$. 
Our intention is to force that both Lam\'{e} parameters $\mu(x)$ and $\lambda(x)$
take the same shape but different scale.
\\
\\
In more detail, we define the quantities $a_{\max}$ and $\tau$ as
 \begin{align}\label{def_a}
 a_{\max}&:=\mu_0-\frac{\mu_0^2}{\mu_0+c^\mu},\\
\tau&:=\frac{1}{a_{\max}}\left(\lambda_0-\frac{\lambda_0^2}{\lambda_0+c^\lambda}\right),\label{def_tau}
 \end{align}
\noindent
such that
\begin{align}\label{cond_a}
-2\left(\mu_0-\frac{\mu_0^2}{\mu}\right)+2a&\leq 0,\\
-\left(\lambda_0-\frac{\lambda_0^2}{\lambda}\right)+\tau a&\leq 0\label{cond_tau}
\end{align}
\noindent
for all $0\leq a\leq a_{\max}$.
\\
\\ 
In addition, we set the residual $r(\nu)$ as
 \begin{align*}
 r(\nu):=\Lambda(\lambda,\mu)-\Lambda(\lambda_0,\mu_0)-\Lambda^\prime(\lambda_0,\mu_0)(\tau \nu,\nu)
 \end{align*}
 \noindent
and the components of the corresponding matrix ${\bf R}(\nu)$ are given by 
\begin{align*}
({\bf R}(\nu))_{i,j=1,\ldots M}:=\left(\langle g_i,r(\nu)g_j\rangle\right)_{i,j=1,\ldots M}.
\end{align*}
 \noindent
{\color{black} We want to remark, that we use the same boundary loads $g_i$, $i=1,\ldots, M,$ as in Section 2.}
\\
\\
 Finally, we introduce the set
 \begin{align*}
 \mathcal{C}:=\bigg\lbrace \nu\in L^\infty_+(\Omega):
\nu&=\sum_{k=1}^p a_k \chi_{k},\,a_k\in\mathbb{R},\,0\leq a_k\leq\min(a_{\max},\beta_k)\bigg\rbrace
 \end{align*}
 \noindent 
with
\begin{align}\label{def_beta}
\beta_k:=\mathrm{max}\left\lbrace a\geq 0:\, \Lambda(\lambda,\mu)-\Lambda(\lambda_0,\mu_0)\leq \Lambda^\prime(\lambda_0,\mu_0)
(\tau a\chi_k,a\chi_k)\right\rbrace,
\end{align}
 {\color{black} 
 \noindent
where we set by $\chi_k:=\chi_{\mathcal{B}_k}$. 
\\
\\
{\color{black}
Note that the set on the right hand side of (\ref{def_beta}) is non-empty since it contains the value zero by Corollary 
\ref{monotonicity} and our assumptions $\lambda\geq \lambda_0, \mu\geq \mu_0$.}
\\
\\
 Then, we modify the original minimization problem (\ref{min_prob_mon_based}) to
 }
 \begin{align*}
 \min_{{\nu}\in\mathcal{C}}\Vert {\bf R}(\nu) \Vert_F. 
 \end{align*}
 
 \begin{remark}
We want to remark that $\beta_k$ is defined  via the infinite-dimensional Neumann-to-Dirichlet operator
$\Lambda(\lambda,\mu)$
and does not involve the finite dimensional matrix $\bf R$. 
For the numerical realization we will require a discrete version $\tilde{\beta}_k$ of $\beta_k$
introduced later on.
\end{remark}
}

 \subsubsection{Main Results}
 
{\color{black}
In the following we present our main results and will show that the choices of the quantities
$a_{\mathrm{max}}$ and $\tau$ will lead the correct reconstruction of the support of 
$\mu(x)$ and $\nu(x)$, which we introduced in (\ref{def_a}) and (\ref{def_tau}), respectively, based
on the lower bounds from the monotonicity tests as stated in (\ref{cond_a}) and (\ref{cond_tau}).
}
\\
 \begin{theorem}\label{theo_min_prob_+}
 Consider the minimization problem
 \begin{align}\label{min_prob_+}
\min_{\nu\in\mathcal{C}}\Vert {\bf R}(\nu)\Vert_F.
 \end{align}
 \noindent
 The following statements hold true:
 \begin{itemize}
 \item[(i)] Problem (\ref{min_prob_+}) admits a unique minimizer $\hat{\nu}$.
 \item[(ii)] $\mathrm{supp}(\hat{\nu})$ and $\mathcal{D}$ agree up to the pixel partition, i.e. for any pixel $\mathcal{B}_k$ 
 \begin{align*}
 \mathcal{B}_k\subset\mathrm{supp}(\hat{\nu})\quad\text{if and only if}\quad \mathcal{B}_k\subset \mathcal{D}.
 \end{align*}
 \noindent
 {\color{black}
 Moreover, 
 \begin{align*}
\hat{\nu}=\sum_{\mathcal{B}_k\subseteq \mathcal{D}} a_{\max}\chi_{k}.
 \end{align*}
 }
 \end{itemize}
 \end{theorem}
 \noindent
 {\color{black}
 Now we deal with noisy data and introduce the corresponding residual
 \begin{align}\label{resid_noise}
 r_\delta(\nu):=\Lambda^\delta(\lambda,\mu)-\Lambda(\lambda_0,\mu_0)-\Lambda^\prime(\lambda_0,\mu_0)(\tau \nu,\nu). 
 \end{align}
 Based on this, ${\bf R}_\delta(\nu)$ represents the matrix 
 $(\langle g_i,r_\delta(\nu)g_j\rangle)_{i,j=1,\ldots M}$.
 \\
 \\
\noindent
Further on, the admissible set for noisy data is defined by
}
 \begin{align*}
 \mathcal{C}_\delta:=\bigg\lbrace \nu\in L^\infty_+(\Omega):
\nu&=\sum_{k=1}^p a_k \chi_{k},\,a_k\in\mathbb{R},\,0\leq a_k\leq\min(a_{\max},\beta_{k,\delta})\bigg\rbrace
 \end{align*}
 \noindent
 with
\begin{align}\label{beta_delta}
\beta_{k,\delta}:=\mathrm{max}\left\lbrace a\geq 0:\, \Lambda^{\delta}(\lambda,\mu)-\Lambda(\lambda_0,\mu_0)-\delta I\leq \Lambda^\prime(\lambda_0,\mu_0)
(\tau a\chi_k,a\chi_k)\right\rbrace.
\end{align}
\noindent
Thus, we present the following stability result.
 {\color{black}
 \begin{theorem}\label{theo_conv_+}
 Consider the minimization problem
 \begin{align}\label{min_prob_delta_+}
\min_{\nu\in\mathcal{C}_\delta}\Vert {\bf R}_\delta(\nu)\Vert_F.
 \end{align}
 \noindent
 The following statements hold true:
 \begin{itemize}
 \item[(i)] Problem (\ref{min_prob_delta_+}) admits a minimizer.
 \item[(ii)] 
 {\color{black}Let $\hat{\nu}=\sum\limits_{\mathcal{B}_k\subseteq \mathcal{D}}a_{\max}\chi_{k}$
 be the minimizer (\ref{min_prob_+}) and 
$\hat{\nu}_\delta=\sum\limits_{k=1}^p a_{k,\delta}\chi_{k}$
 of problem (\ref{min_prob_delta_+}), respectively. Then $\hat{\nu}_\delta$ converges pointwise and uniformly to 
 $\hat{\nu}$ as $\delta$ goes to $0$.
 }
 \end{itemize}
 \end{theorem}
 }
 \noindent
 \\
 {\color{black}
 \begin{remark}
 In \cite{Eberle_mon_test}, we used monotonicity methods to solve the inverse problem of shape reconstruction.
 In Theorem 1 and Theorem 2, we applied the same monotonicity methods to construct constraints for the 
 residual based inversion technique. Both methods have a rigorously proven convergence theory, however
 the monotonicity-based regularization approach turns out to be more stable regarding noise.
 \end{remark}
 }
 
 \subsubsection{Theoretical Background}
In order to prove Theorem \ref{theo_min_prob_+} as well as Theorem \ref{theo_conv_+}, we have to take
a look at the following.
\\
 \begin{lemma}\label{pixel_+}
{\color{black}Let $a_{\mathrm{max}}$ and $\tau$ be defined as in {\color{black}(\ref{def_a}) and (\ref{def_tau})}, respectively,
$\lambda,\mu\in L^\infty_+(\Omega)$ and we assume that $\lambda\geq\lambda_0$, $\mu\geq \mu_0$,
where $\lambda_0,\mu_0$ are constant.
Then we have for any pixel $\mathcal{B}_k$, $\mathcal{B}_k\subseteq \mathcal{D}$ if and only if 
 $\beta_{k}>0$, where $\beta_k$ is defined in (\ref{def_beta})}.
 \end{lemma}
 \begin{proof}
 We adopt the proof of Lemma 3.4 from \cite{harrach2016enhancing}.
 \\
 \\
 \noindent
 {\it{Step 1:}} {\color{black} First, we verify that from $\mathcal{B}_k\subseteq \mathcal{D}$ it follows that
$\beta_{k}>0$.}
\\
 In fact, by applying the monotonicity principle (\ref{mono_lame}) {\color{black} multiplied by $-1$}
 \noindent
for
 \begin{align*}
\lambda_1:=\lambda ,\mu_1:=\mu \,\,\,\mathrm{and}\,\,\,
{\color{black} \lambda_2:=\lambda_0,\mu_2:=\mu_0,}
 \end{align*}
 \noindent
 we end up with the following inequalities
for all pixel $\mathcal{B}_k$, all $ a\in[0,a_{\max}]$ and all $g\in L^2(\Gamma_{\textup{N}})^d$
 \begin{align*}
 &\langle g,\left(\Lambda(\lambda,\mu)-\Lambda(\lambda_0,\mu_0)-
 \Lambda^{\prime}(\lambda_0,\mu_0)(\tau a\chi_{k},a\chi_{k})\right)g\rangle\\
 \leq& -\int_{\Omega} 2\left(\mu_0-\frac{\mu_0^2}{\mu}\right)\hat{\nabla}u_0^{g}:\hat{\nabla}u_0^{g}{\color{black} \,dx}+
 {\color{black} \int_{\Omega}}2 a\chi_{k}\hat{\nabla}u_0^{g}:\hat{\nabla}u_0^{g}\,dx\\
 &-\int_{\Omega}\left(\lambda_0-\frac{\lambda_0^2}{\lambda}\right)\nabla\cdot u_0^{g} \nabla \cdot u_0^{g}\,dx
 +{\color{black} \int_{\Omega}}\tau a\chi_{k}\nabla\cdot u_0^{g} \nabla \cdot u_0^{g}\,dx\\
 \leq& -\int_{\mathcal{D}} 2\left(\mu_0-\frac{\mu_0^2}{\mu}\right)\hat{\nabla}u_0^{g}:\hat{\nabla}u_0^{g}\,dx
 +\int_{\mathcal{B}_k}2 a_{\max}\hat{\nabla}u_0^{g}:\hat{\nabla}u_0^{g}\,dx\\
 &-\int_{\mathcal{D}}\left(\lambda_0-\frac{\lambda_0^2}{\lambda}\right)\nabla\cdot u_0^{g} \nabla \cdot u_0^{g}\,dx
 {\color{black}
 +\int_{\mathcal{B}_k}\tau a_{\max}\nabla\cdot u_0^{g} \nabla \cdot u_0^{g}\,dx}\\
  \leq & 0.
 \end{align*}
 \noindent
{\color{black}
In the above inequalities, we used the shorthand notation $u_0^{g}$ for the unique solution $u_{(\lambda_0,\mu_0)}^{g}$.}
 The last inequality holds due to the fact that
 $a_{\max}$ and {\color{black} $\tau$} fulfill
\begin{align*}
-2\left(\mu_0-\frac{\mu_0^2}{\mu}\right)+2a_{\max}&\leq 0,\\
 {\color{black} -\left(\lambda_0-\frac{\lambda_0^2}{\lambda}\right)+\tau a_{\max}}&\leq 0
\end{align*}
 \noindent
 in $\mathcal{D}$ and that $\mathcal{B}_k$ lies inside $\mathcal{D}$.
 \noindent
 \\
 \\
 {\color{black} We want to remark, that compared with the corresponding proof in \cite{harrach2016enhancing},
 this shows us that we require conditions on $a_{\max}$ as well as on $\tau$ (c.f. Equation (\ref{cond_a}) and 
 (\ref{cond_tau})) due to the fact that we
 deal with two unknown parameters ($\lambda$ and $\mu$) instead of one.}
 \noindent
 \\
 \\
  {\it{Step 2:}} 
 {\color{black} In order to prove the other direction of the statement, let $\beta_{k}> 0$.
We will show that $\mathcal{B}_k\subseteq \mathcal{D}$ by contradiction.}
 \\
 Assume that $\mathcal{B}_k\not\subseteq \mathcal{D}$ and
 $\beta_{k}>0$.
 {\color{black}
 Applying the monotonicity principle from Lemma \ref{mono},
 \begin{align*}
 \Lambda(\lambda,\mu)-\Lambda(\lambda_0,\mu_0)\geq \Lambda^{\prime}(\lambda_0,\mu_0)((\lambda,\mu)-(\lambda_0,\mu_0)),
 \end{align*}
 }
 \noindent
{\color{black} with the definition of $\beta_k$ in (\ref{def_beta}), we are led to
\begin{align*}
0&\geq \Lambda(\lambda,\mu)-\Lambda(\lambda_0,\mu_0)-\Lambda^{\prime}(\lambda_0,\mu_0)(\tau\beta_k\chi_k,\beta_k\chi_k)\\
&\geq \Lambda^{\prime}(\lambda_0,\mu_0)((\lambda,\mu)-(\lambda_0,\mu_0))- \Lambda^{\prime}(\lambda_0,\mu_0)(\tau\beta_k\chi_k,\beta_k\chi_k).
\end{align*}
}
 {\color{black} Based on this, we conclude that} for all $g\in L^2(\Gamma_{\textup{N}})^d$
 \begin{align}\label{mono_pixel}
 &\int_{\mathcal{B}_k}\tau \beta_{k}\nabla\cdot u_0^{g} \nabla \cdot u_0^{g}\,dx
 +2\int_{\mathcal{B}_k}\beta_{k}\hat{\nabla} u_0^{g} :\hat{\nabla} u_0^{g}\,dx\\ \nonumber
 &\leq \int_{\Omega}(\lambda-\lambda_0)\nabla\cdot u_0^{g} \nabla \cdot u_0^{g}\,dx
 +2\int_{\Omega}(\mu-\mu_0)\hat{\nabla} u_0^{g} : \hat{\nabla} u_0^{g}\,dx\\ \nonumber
  &\leq \int_{\mathcal{D}} C^{\lambda} \nabla\cdot u_0^{g} \nabla \cdot u_0^{g}\,dx
 +2\int_{\mathcal{D}}C^{\mu}\hat{\nabla}u_0^{g} : \hat{\nabla} u_0^{g}\,dx.
 \end{align}
 \noindent
 On the other hand, {\color{black} using} the localized potentials
 {\color{black}
 in a similar procedure as in the proof of Theorem 2.1 in 
 \cite{Eberle_mon_test},
 }
 we can find a sequence $(g_m)_{m\in\mathbb{N}}\subset L^2(\Gamma_{\textup{N}})^d$ such that the 
 solutions $(u_0^m)_{m\in\mathbb{N}}\subset H^1(\Omega)^d$ of the forward problem
 \noindent
  (when the Lam\'{e} parameter are chosen to be $\lambda_0$, $\mu_0$ and the boundary forces $g=g_m$) fulfill 
 \begin{align*}
 &\lim_{m\to\infty}\int_{\mathcal{B}_k}\hat{\nabla}u_0^{m}:\hat{\nabla}u_0^{m}\,dx=\infty,\quad\,\,\,
 \lim_{m\to\infty}\int_{\mathcal{D}}\hat{\nabla}u_0^{m}:\hat{\nabla}u_0^{m}\,dx=0,\\
 &\lim_{m\to\infty}\int_{\mathcal{B}_k}\nabla\cdot u_0^{m} \nabla\cdot u_0^{m}\,dx=\infty,\quad
 \lim_{m\to\infty}\int_{\mathcal{D}}\nabla\cdot u_0^{m} \nabla\cdot u_0^{m}\,dx=0,
 \end{align*}
 \noindent
 which contradicts (\ref{mono_pixel}).
 \end{proof}
\noindent

{\color{black}
\noindent
 \begin{lemma}\label{pos_S}
 For all pixels $\mathcal{B}_k$, denote by ${\bf S}_k^\tau$ the matrix 
\begin{align*}
 {\bf S}_k^\tau:={\color{black}\left(\langle g_i,-\Lambda^\prime(\lambda_0,\mu_0)(\tau\chi_k,\chi_k)g_j\rangle\right)_{i,j=1,\ldots,M}}.
\end{align*}
\noindent
 Then ${\bf S}_k^\tau$ is a positive definite matrix.
 \end{lemma}
 }
 \begin{proof}
 We adopt the proof of Lemma 3.5 from \cite{harrach2016enhancing} for the matrix
 ${\bf S}_k^\tau$, which directly yields the desired result.
 \end{proof}
 \noindent
 \begin{proof}{{\bf (Theorem \ref{theo_min_prob_+})}} This proof is based on the proof of Theorem 3.2 from \cite{harrach2016enhancing}.
 \\
to (i) Since the functional
\vspace{-0.3cm}
 \begin{align*}
 \nu\mapsto \Vert {\bf R}(\nu)\Vert_F^2:=\sum_{i,j=1}^M\langle g_i,r(\nu)g_j\rangle^2
 \end{align*}
 \noindent
 is continuous, it admits a minimizer in the compact set $\mathcal{C}$. 
 \\
 The uniqueness of $\hat{\nu}$ will follow from the proof of (ii) {\it{Step 3}}.
 \\
 \\
 to (ii) {\it{Step 1}} We shall check that for all
 \begin{align*}
 \nu&=\sum_{k=1}^p a_k\chi_k\quad\text{satisfying}\quad 0\leq a_k\leq\min(a_{\max},\beta_k),
 \end{align*}
 \noindent
 it holds that  {\color{black} $r(\nu)\leq 0$} in quadratic sense. 
 {\color{black}
 We want to remark that for
 $a_{\max}$ and  {\color{black} $\tau$} it holds that
(\ref{cond_a}) and (\ref{cond_tau})
in $\mathcal{D}$.
}
 \noindent
\\
{\color{black} We proceed similar as in the proof of Lemma \ref{pixel_+} and use
Lemma \ref{mono_2} for $\lambda_1=\lambda, \mu_1=\mu, \lambda_2=\lambda_0$ and
$\mu_2=\mu_0$. In addition, we multiply the whole expression with $-1$.
Thus, it holds that}
 \noindent
 \begin{align*}
 \langle g&,(\Lambda(\lambda,\mu)-\Lambda(\lambda_0,\mu_0)-\Lambda^\prime(\lambda_0,\mu_0)(\tau \nu,\nu))g\rangle\\
 \leq &-\int_{\mathcal{D}}{\color{black} 2}a_{\max}\hat{\nabla}u_0^{g}:\hat{\nabla}u_0^{g}\,dx+\sum_{k=1}^p\int_{\mathcal{B}_k}{\color{black} 2}a_k\hat{\nabla}u_0^{g}:\hat{\nabla}u_0^{g}\,dx\\
 & {\color{black} -\int_{\mathcal{D}}\tau a_{\max}\nabla\cdot u_0^{g} \nabla\cdot u_0^{g}\,dx}+\sum_{k=1}^p\int_{\mathcal{B}_k}\tau a_k\nabla\cdot u_0^{g} \nabla\cdot u_0^{g}\,dx
 \end{align*}
 \noindent
 for any $g\in L^2(\Gamma_{\textup N})^d$.
\\
\\
\noindent
 If $a_k>0$, {\color{black} it follows $\beta_k\geq a_k>0$, so that Lemma \ref{pixel_+} implies that}
 $\mathcal{B}_k\subseteq\mathcal{D}$. {\color{black} Since $a_k\leq a_{\max}$, we end up with
 $\langle g,r(\nu) g\rangle \leq 0$} for $g\in L^2(\Gamma_{\textup N})^d$.
 \\
 \\
 {\it{Step 2:}}  {\color{black} Let $\hat{\nu}=\sum\limits_{k=1}^p \hat{a}_k\chi_k$} be a minimizer of problem 
 (\ref{min_prob_+}). We {\color{black} show} that $\mathrm{supp}(\hat{\nu})\subseteq \mathcal{D}$.
 \noindent
 \\
 \\
 {\color{black}
 Per definition of $\beta_k$, it holds that $\beta_k\geq \hat{a}_k$. This implies $\beta_k>0$. 
 With Lemma \ref{pixel_+} we have $\mathcal{B}_k\subseteq \mathcal{D}$.
 }
 \\
 \\
 {\it{Step 3:}} 
  {\color{black} We will prove} that, if $\hat{\nu}=\sum\limits_{k=1}^p \hat{a}_k\chi_k$ is a minimizer of 
 problem (\ref{min_prob_+}), 
 {\color{black}then the representation of $\hat{a}_k$ is given by
 \begin{align*}
\hat{a}_k=
\begin{cases}
 0&\text{for}\,\,\mathcal{B}_k\not\subseteq\mathcal{D},\\
 a_{\mathrm{max}}&\text{for}\,\,\mathcal{B}_k\subseteq\mathcal{D}.
\end{cases}
 \end{align*}
 \noindent 
 {\color{black}In fact}, it holds that $\hat{a}_k< a_{\max}$.
 If there exists a pixel $\mathcal{B}_k$ such that $\hat{\nu}(x)<\min (a_{\max},\beta_k)$ in $\mathcal{B}_k$, 
 we can choose $h^\nu> 0$, such that
 $\hat{\nu}+h^\nu\chi_k=a_{\max}$ in $\mathcal{B}_k$.
We will {\color{black} show} that then,
 \begin{align*}
 \Vert {\bf R}(\hat{\nu}+h^\nu\chi_k)\Vert_F < \Vert {\bf R}(\hat{\nu})\Vert_F,
 \end{align*}
 \noindent
 which contradicts the minimality of $\hat{\nu}$. Thus, it follows that
 $\hat{a}_k=\mathrm{min}\left(a_{\mathrm{max}},\beta_k\right)$.
 }
 \\
 \\
 {\color{black}
 To show the contradiction, let
 \noindent
 $\theta_1(\hat{\nu})\geq \theta_2(\hat{\nu})\geq ...\geq \theta_{M}(\hat{\nu})$ be $M$ eigenvalues of 
 ${\bf R}(\hat{\nu})$ and
 $\theta_1(\hat{\nu}+h^\nu\chi_k)\geq \theta_2(\hat{\nu}+h^\nu\chi_k)\geq ...\geq 
 \theta_{M}(\hat{\nu}+h^\nu\chi_k)$ $M$ eigenvalues of ${\bf R}(\hat{\nu}+h^\nu\chi_k)$.
 \\
 \\
 Since ${\bf R}(\hat{\nu})$ and ${\bf R}(\hat{\nu}+h^\nu\chi_k)$ 
 are both symmetric, all of their eigenvalues are real. By the definition of the Frobenius norm, we {\color{black}obtain} 
 \noindent
 \begin{align*}
 &\Vert {\bf R}(\hat{\nu}+h^\nu\chi_k)\Vert_F^2-\Vert {\bf R}(\hat{\nu})\Vert_F^2\\
 &= \sum_{i=1}^M\vert \theta_i(\hat{\nu}+h^\nu{\color{black}\chi_k})\vert^2-\sum_{i=1}^M\vert\theta_i(\hat{\nu})\vert^2\\
 &=\sum_{i=1}^M \left(\theta_i\left(\hat{\nu}+h^\nu\chi_k)+\theta_i(\hat{\nu}\right)\right)\cdot 
 \left(\theta_i(\hat{\nu}+h^\nu\chi_k)-\theta_i(\hat{\nu})\right).
 \end{align*}
 \noindent
 Due to  {\it{Step 1}}, $r(\hat{\nu})\leq 0$ and $r(\hat{\nu}+h^\nu\chi_k)\leq 0$ in the quadratic sense. 
 Thus, for all $x=(x_1,...,x_{M})^T\in\mathbb{R}^M$, we have
 \begin{align*}
 x^T {\bf R}(\hat{\nu})x=\sum_{i,j=1}^M x_i x_j \langle g_i, r(\hat{\nu})g_j\rangle=
 \langle g,r(\hat{\nu}) g\rangle\leq 0,
 \end{align*}
 \noindent
 where $ g=\sum\limits_{i=1}^M x_i g_i$. This means that $-{\bf R}(\hat{\nu})$ is a positive semi-definite symmetric matrix in $\mathbb{R}^{M\times M}$.
{\color{black}Due to the fact, that all eigenvalues of a positive semi-definite symmetric matrix are non-negative, it follow that}
 $\theta_i(\hat{\nu})\leq 0$ for all
 $i\in\lbrace 1, ..., M\rbrace$.
{\color{black}By the same considerations,} $-{\bf R}(\hat{\nu}+h^\nu\chi_k)$ is also a positive semi-definite matrix. 
 }
 \noindent
We want to remark, that ${\bf S}_k^\tau$ is positive definite as proven
 in Lemma \ref{pos_S} and hence, all $M$ eigenvalues of
$\theta_1({\bf S}_k^{\tau})
\geq ...\geq 
\theta_{M}({\bf S}_k^{\tau})$ 
are positive. Since 
 \begin{align*}
 {\bf R}(\hat{\nu}+h^\nu\chi_k)=
 {\bf R}(\hat{\nu})+h^\nu {\bf S}_k^{\tau}
 \end{align*}
 \noindent
 and the matrices ${\bf R}(\hat{\nu}+h^\nu\chi_k)$, ${\bf R}(\hat{\nu})+h^\nu$ and
  ${\bf S}_k^\tau$ are symmetric, we can apply Weyl's Inequalities to get
 \begin{align*}
 \theta_i(\hat{\nu}+h^\nu\chi_k)\geq \theta_i(\hat{\nu})
 +\theta_{M}(h^\nu {\bf S}_k^{\tau})
 >\theta_i (\hat{\nu})
 \end{align*}
 \noindent
 for all $i\in\lbrace 1,...,M\rbrace.$
 \\
 \\
 {\color{black}
 In summary we end up with
 \begin{align*}
 \Vert {\bf R}(\hat{\nu}+h^\nu \chi_k)\Vert_F < \Vert {\bf R}(\hat{\nu})\Vert_F,
 \end{align*}
 \noindent
which contradicts the minimality of $\hat{\nu}$ and thus, ends the proof of {\it{Step 3}}.
}
 \\
 \\
 {\color{black}
 {\it{Step 4:}} {\color{black}We show that,} if $\mathcal{B}_k\subseteq \mathcal{D}$, then $\mathcal{B}_k\subseteq\mathrm{supp}(\hat{\nu})$. 
 Indeed, since $\hat{\nu}$ is a minimizer of problem (\ref{min_prob_+}), {\it{Step 3}} implies that 
 \begin{align*}
 \hat{\nu}=\sum_{k=1}^p \min(a_{\max},\beta_k)\chi_k.
 \end{align*}
 }
 \\
 \noindent
 Since $\mathcal{B}_k\subseteq\mathcal{D}$, it follows from Lemma \ref{pixel_+} that
 $\min(a_{\max},\beta_k)>0$. Thus, $\mathcal{B}_k\subseteq \mathrm{supp}(\hat{\nu})$.
 \\
 \\
 \noindent
 {\color{black}
 In conclusion, problem (\ref{min_prob_+}) admits a unique minimizer $\hat{\nu}$ with
 \begin{align*}
\hat{\nu}=\sum_{k=1}^p \min(a_{\max},\beta_k)\chi_k.
 \end{align*}
 \noindent
 This minimizer fulfills
 \begin{align*}
\hat{\nu}=
 \begin{cases}
a_{\max}\,\,\text{in}\,\,\mathcal{B}_k,&\quad\text{if}\,\, \mathcal{B}_k\,\,\text{lies inside}\,\,\mathcal{D},\\
 0\,\,\hspace{0.55cm}\text{in}\,\, \mathcal{B}_k,&\quad\text{if}\,\,\mathcal{B}_k\,\,\text{does not lie inside}\,\,\mathcal{D},
 \end{cases}
 \end{align*}
 \noindent
 so that 
 \begin{align*}
  \hat{\nu}=\sum_{\mathcal{B}_k\subseteq \mathcal{D}} a_{\mathrm{max}}\chi_k.
 \end{align*}
 }
 \end{proof}
\noindent
Next, we go over to noisy data and take a look at the following lemma,
{\color{black}where we set $V^\delta:=\frac{1}{2}( V^\delta + (V^\delta)^{*}) $, since we always can redefine the data $V^\delta$ in this way without loss of generality. Thus, we can assume that $V^\delta$ is self-adjoint.}
\\
 \begin{lemma}\label{lem_delta_+}
 Assume that $\Vert \Lambda^\delta(\lambda,\mu)-\Lambda(\lambda,\mu)\Vert \leq \delta$. 
 Then for every pixel $\mathcal{B}_k$, it holds that
 $\beta_k\leq \beta_{k,\delta}$ for all $\delta >0$.
 \end{lemma}
 
 \begin{proof}
 The proof follows the lines of Lemma 3.7 in \cite{harrach2016enhancing} with the following modifications.
 We have to check that $\beta_k$ {\color{black} as given in (\ref{def_beta})} fulfills the relation
 \begin{align*}
 \vert V^\delta\vert +\Lambda^\prime(\lambda_0,\mu_0)(\tau  a \chi_k,a \chi_k)\geq -\delta I
 \quad \text{for all}\,\,  a \in[0,\beta_k],
 \end{align*}
 \noindent
{\color{black} where $\vert V^\delta\vert=\sqrt{(V^\delta)^*V^\delta}$.}
\\
\\
 As proven in \cite{harrach2016enhancing}, the operator $V-V^\delta\geq -\delta I$ in quadratic sense. 
 Further on Lemma 3.6 from \cite{harrach2016enhancing} implies $\vert V^\delta\vert\geq V^\delta$, 
{\color{black} since $V^\delta$ is self-adjoint.}  Hence,
 \begin{align*}
 \vert V^\delta\vert&+\Lambda^\prime(\lambda_0,\mu_0)(\tau\beta_k\chi_k,\beta_k\chi_k)\\
 &\geq V^\delta +\Lambda^\prime(\lambda_0,\mu_0)(\tau\beta_k\chi_k,\beta_k\chi_k)\\
 &=V+\Lambda^\prime(\lambda_0,\mu_0)(\tau\beta_k\chi_k,\beta_k\chi_k)+V^\delta-V\\
 &\geq -\delta I.
\end{align*}   
 \end{proof}
 \begin{remark}
 As a consequence, it holds that
 \begin{itemize}
 \item[1.] {\color{black} If $\mathcal{B}_k$ lies inside $\mathcal{D}$, then $\beta_{k,\delta}\geq a_{\max}$.}
 \item[2.] If $\beta_{k,\delta}=0$, then $\mathcal{B}_k$ does not lie inside $\mathcal{D}$. 
 \end{itemize}
 \end{remark}
{\color{black}
 \begin{proof}{{\bf (Theorem \ref{theo_conv_+})}} This proof is based on {\color{black} the proof of Theorem 3.8 in} \cite{harrach2016enhancing}.
 \\
 to (i) 
 For the proof of the existence of a minimizer of (\ref{min_prob_delta_+}), we argue as in the proof of 
 Theorem \ref{theo_min_prob_+} (i). 
{\color{black} First, we take a look at the functional
 \begin{align}\label{func_R_delta}
  \nu\mapsto \Vert {\bf R}_\delta(\nu)\Vert_F^2,
 \end{align}
 \noindent
which is defined by $({\bf R}_{\delta}(\nu))_{i,j=1,\ldots M}:=\left(\langle g_i,r_{\delta}(\nu)g_j\rangle\right)_{i,j=1,\ldots M}$ via the residual (\ref{resid_noise}).
Since the functional (\ref{func_R_delta})}
 is continuous, it follows that there exists at least one minimizer in the compact set 
 $\mathcal{C}^\delta$.
 \\
 \\
 \noindent
 to (ii) {\it Step 1:} Convergence of a subsequence of $\hat{\nu}_\delta$
 \\
 For any fixed $k$, the sequence
 $\lbrace\hat{a}_{k,\delta}\rbrace_{\delta>0}$  is
 bounded from below by $0$ and from above by $a_{\max}$, respectively.  
 By Weierstrass' Theorem, there exists a subsequence 
 $(\hat{a}_{1,\delta_n},...,\hat{a}_{p,\delta_n})$ converging to some limit 
 $(a_{1},...,a_p)$.
 Of course, $0\leq  a_k\leq a_{\max}$ for all $k=1,...,p$.
 \\
 \\
{\it Step 2:} Upper bound and limit
 \\
We shall check that $a_k\leq \beta_k$ for all $k=1,...,p$.
As shown in the proof of Theorem 3.8 in \cite{harrach2016enhancing}, $\vert V^\delta\vert$ converges to $\vert V\vert$
{\color{black}
in the operator norm as $\delta$ goes to $0$,} and hence, for any fixed $k$,
\begin{align*}
\vert V\vert +\Lambda^\prime(\lambda_0,\mu_0)(\tau a_k\chi_k,a_k\chi_k)=
\lim_{\delta_n\to 0}(\vert V^{\delta_n}\vert +\Lambda^\prime(\lambda_0,\mu_0)(\tau\hat{a}_{k,\delta_n}\chi_k,\hat{a}_{k,\delta_n} \chi_k))
\end{align*}
\noindent
in the operator norm. As in \cite{harrach2016enhancing}, {\color{black} we obtain that for all $g\in L^2(\Gamma_{\mathrm{N}})^3$,}
\begin{align*}
\langle g, (\vert V\vert +\Lambda^\prime(\lambda_0,\mu_0)(\tau a_k\chi_k,a_k\chi_k))g\rangle\geq 0.
\end{align*}
{\it Step 3:} Minimality of the limit
\\
{\color{black}Due to Lemma \ref{lem_delta_+}, we know that}
$\min(a_{\max},\beta_k)\leq\min(a_{\max},\beta_{k,\delta})$ for all $k=1,...,p$. 
Thus, $\hat{\nu}$ belongs to the admissible set $\mathcal{C}_\delta$ of the 
minimization problem (\ref{min_prob_delta_+}) for all $\delta>0$. 
By minimality of $\hat{\nu}_\delta$, we obtain
\begin{align*}
\Vert {\bf R}_\delta(\hat{\nu}_\delta)\Vert_F\leq \Vert {\bf R}_\delta(\hat{\nu})\Vert_F.
\end{align*}
\noindent
Denote by $\nu=\sum\limits_{k=1}^p a_k\chi_k$, where $a_k$ are 
the limits {\color{black}derived in} {\it{Step 1}}. We have that
\begin{align*}
\Vert {\bf R}_{\delta_n}(\hat{\nu}_{\delta_n})\Vert_F^2&=
\sum_{i,j=1}^M\left\langle  g_i,\left(-V^{\delta_n}-\sum_{k=1}^p \Lambda^\prime(\lambda_0,\mu_0)
(\tau \hat{a}_{k,\delta_n}\chi_k,\hat{a}_{k,\delta_n}\chi_k)\right) g_j\right\rangle^2,\\
\Vert {\bf R}({\nu})\Vert_F^2&=
\sum_{i,j=1}^M\left\langle  g_i,\left(-V-\sum_{k=1}^p \Lambda^\prime(\lambda_0,\mu_0)(\tau a_k\chi_k,a_k\chi_k)\right) g_j\right\rangle^2.
\end{align*}
\noindent
With the same arguments as in the proof of Theorem 3.8 in \cite{harrach2016enhancing}, 
{\color{black} i.e. that $V$ converges to $V^\delta$ as well as $\hat{a}_{k,\delta}$ goes to $a_k$, we are led to}
\begin{align*}
\Vert {\bf R}(\nu)\Vert_F\leq \Vert {\bf R}(\hat{\nu})\Vert_F.
\end{align*}
\noindent
Further on, by the uniqueness of the minimizer we obtain $\nu=\hat{\nu}$ that is 
\begin{align*}
 a_k=\hat{a}_k=\begin{cases}
 0&\text{for}\,\,\mathcal{B}_k\not\subseteq\mathcal{D},\\
 a_{\mathrm{max}}&\text{for}\,\,\mathcal{B}_k\subseteq\mathcal{D}.
\end{cases}
\end{align*}
\noindent
\\
{\it Step 4:} Convergence of the whole sequence $\hat{\nu}_\delta$
\\
Again this is obtained in the same way as in  \cite{harrach2016enhancing} {\color{black} and is based on the
knowledge that every subsequence of $(\hat{a}_{1,\delta},\ldots,\hat{a}_{p,\delta})$ possesses a convergent subsubsequence,
that goes to the limit $(\min(a,\beta_1),\ldots,$ $\min(a,\beta_p))$.}
 \end{proof}
 \noindent
\begin{remark}
 All in all, we are led to the discrete formulation of the minimization problem for noisy data: 
\begin{align}\label{min_prob_disc}
\min_{{\nu}\in \mathcal{C}_\delta}\Vert {\bf R}_\delta(\nu)\Vert_F,
\end{align}
 \noindent 
 under the constraint
  \begin{align}
   0 &\leq \nu_k \leq \min\left(a_{\max}, \tilde{\beta}_{k,\delta}\right),
  \end{align}
  \noindent
  where 
 \begin{align}
 a_{\max}&=\mu_0-\frac{\mu_0^2}{\mu_0+c^\mu},\\
 \tau&=\dfrac{1}{a_{\mathrm{max}}}\left(\lambda_0-\dfrac{\lambda_0^2}{\lambda_0+c^{\lambda}}\right),\\
 \tilde{\beta}_{k,\delta}&=\max\lbrace a \geq 0\,:\, - a {\bf S}^{\tau}_k\geq -\delta {\bf I}-\vert {\bf V}^\delta\vert \rbrace
 \end{align} 
\noindent
{\color{black}
 with $\vert {\bf V}^\delta\vert:=\sqrt{({\bf V}^\delta)^{*} {\bf V}^\delta}$.
\\
\\
We want to mention, that $\bf{V}$ is positive definite, however, $\bf{V}^\delta$ is not in general, which leads to problems in the proofs. Hence, we use $\vert {\bf V}^\delta\vert$ instead.}
 \end{remark}
}
 
 \noindent
 \\
 Next, we take a closer look at the determination of $\tilde{\beta}_{k,\delta}$ (see \cite{harrach2016enhancing}),
 where $\tilde{\beta}_{k,0}=\tilde{\beta}_{k}$:
 \\
 \\
 First, we replace the infinite-dimensional operators $\vert { V}^\delta\vert$ and $\Lambda^\prime(\lambda_0,\mu_0)$ 
{\color{black} in (\ref{beta_delta})} by
 the $M\times M$ matrices ${\bf V}^\delta$, ${\bf S}_k^\tau$ such that we need to find $\tilde{\beta}_{k,\delta}$ with
 \begin{align*}
 - a {\bf S}_k^\tau &\geq -\delta {\bf I} -\vert {\bf V}^\delta \vert
 \end{align*}
 \noindent
 for all $ a \in [0,\tilde{\beta}_{k,\delta}]$.
 \noindent
 Due to the fact that $\delta {\bf I} +\vert {\bf V}^\delta \vert$ is a Hermitian positive-definite matrix, the Cholesky decomposition allows us to decompose 
 it into the product of a lower triangular matrix and its conjugate transpose, i.e.
 \begin{align*}
 \delta {\bf I} +\vert {\bf V}^\delta \vert = {\bf L} {\bf L}^T.
 \end{align*}
 \noindent
 We want to remark that this decomposition is unique. In addition, ${\bf L}$ is invertible, since
 \begin{align*}
 0< \det(\delta {\bf I}+{\bf V}^\delta)=\det({\bf L})\det({\bf L}^T)=\det({\bf L})\overline{\det({\bf L})}.
 \end{align*}
 \noindent
 For each $a>0$, it follows that
 \begin{align*}
 -a{\bf S}_k^\tau +  \delta {\bf I} +\vert {\bf V}^\delta \vert &=-a {\bf S}_k^\tau + {\bf L} {\bf L}^T = 
 {\bf L} (-a {\bf L}^{-1}{\bf S}_k^\tau({\bf L}^T)^{-1}+{\bf I}){\bf L}^T.
 \end{align*}
 \noindent
 It should be noted that in this notation
 $\tilde{\beta}_{k,M,\delta}=\tilde{\beta}_{k-M,\delta}$ for $k=M+1,...,2M$.
 \\
 Based on this, we go over to the consideration of the eigenvalues and apply Weyl's Inequality. 
 {\color{black}
 Since the positive semi-definiteness  of
 $-a {\bf S}_k^\tau +\delta {\bf I} +\vert {\bf V}^\delta \vert$  is
 equivalent to the positive semi-definiteness of 
  $a {\bf L}^{-1}{\bf S}_k^\tau({\bf L}^T)^{-1}+{\bf I}$, we obtain
 \begin{align*}
 \theta_j(-a{\bf L}^{-1} {\bf S}_k^\tau({\bf L}^T)^{-1}+{\bf I})= a\theta_j(-{\bf L}^{-1} {\bf S}_k^\tau({\bf L}^T)^{-1})+1,\quad j=1,...,M,
 \end{align*}
 \noindent
 where $\theta_1(A)\geq ... \geq \theta_{M}(A)$ denote the $M$-eigenvalues of some matrix $A$.
 }
 \\
 \\
 \noindent
 Further, let 
 $\overline{\theta}_{M}({\bf L}^{-1} {\bf S}_k^\tau ({\bf L}^T)^{-1})$ be the smallest eigenvalue of the matrix 
 ${\bf L}^{-1} {\bf S}^\tau_k ({\bf L}^T)^{-1}$. Since ${\bf S}^\tau_k$ is positive semi-definite, so is
 ${\bf L}^{-1} {\bf S}^\tau_k ({\bf L}^T)^{-1}$. Thus, $\overline{\theta}_{M}({\bf L}^{-1} {\bf S}^\tau_k ({\bf L}^T)^{-1})\leq 0$.
 Following the lines of \cite{harrach2016enhancing}, we obtain
\noindent
 \begin{align}\label{calc_beta_delta}
 \tilde{\beta}_{k,\delta}=-\frac{1}{\overline{\theta}_{M}({\bf L}^{-1} {\bf S}_k^\tau ({\bf L}^T)^{-1})}\geq 0.
 \end{align}

 \subsubsection{Numerical Realization}
 
 We close this section with a numerical example, where we again consider two inclusions (tumors)
 in a biological tissue as shown in Figure \ref{standard} (for the values of the Lam\'e parameter see 
 Table \ref{lame_parameter_mono}). In addition to the Lam\'e parameters,
we use the estimated lower and upper bounds $c^{\lambda},c^{\mu},C^{\lambda},C^{\mu}$ given in Table \ref{bounds}.

 \begin{table} [H]
 \begin{center}
 \begin{tabular}{ |c|c| c |}  
\hline
            & $\gamma^{\lambda}$ & $\gamma^{\mu}$ \\
  \hline
lower bounds &  $c^{\lambda}=1.2\cdot 10^6$   &  $c^{\mu}=1.2\cdot 10^4$   \\
 \hline
upper bounds &  $C^{\lambda}=1.7\cdot 10^6$ &  $C^{\mu}=1.7\cdot 10^4$  \\
\hline
\end{tabular}
\end{center}
\caption{Lower and upper bounds $c^{\lambda},c^{\mu},C^{\lambda},C^{\mu}$ in [Pa].}
\label{bounds}
\end{table}
\noindent
For the implementation, we again consider difference measurements and apply {\it quadprog} from Matlab in order to 
solve the minimization problem. {\color{black} In more detail, we perform the following steps:
\\
\begin{itemize}
\item[\textbf{1.)}] Calculate
\begin{align*}
\langle(\Lambda(\lambda_0,\mu_0)-\Lambda^\delta(\lambda,\mu))g_i,g_j\rangle_{i,j=1,\cdots,M}
\end{align*}
\noindent
with COMSOL to obtain ${\bf V}$ via (\ref{monotonicity_test_1}).
\item[\textbf{2.)}] Evaluate $\hat{\nabla}u^{g_i}_{(\lambda_0,\mu_0)}$ and $\nabla\cdot u^{g_i}_{(\lambda_0,\mu_0)}$ for 
$i=1,\cdots,M$, in Gaussian nodes for each tetrahedron.
\item[\textbf{3.)}] Calculate ${\bf S^\lambda}, {\bf S^\mu}$ (cf. Equations (\ref{S_lam}) and (\ref{S_m})) via Gaussian quadrature.
\\
\\
Note that ${\bf S^\lambda}$, ${\bf S^\mu}$ can also be calculated from the stiffness matrix of the FEM implementation without additional quadrature errors by the approach given in \cite{Harrach_FEM}.
\item[\textbf{4.)}] Calculate ${\bf S^\tau}={\bf S^\mu}+\tau{\bf S^\lambda}$ with $\tau$ as in (\ref{def_tau}).
\item[\textbf{5.)}] Calculate $\tilde{\beta}^\delta_k$, $k=1,\ldots,p,$ as in (\ref{calc_beta_delta}).
\item[\textbf{6.)}] Solve the minimization problem (\ref{min_prob_delta_+}) with ${\bf R}_{\delta}(\nu)={\bf S^\tau}\nu-{\bf V}^{\delta}$ with {\it quadprog} in Matlab to obtain 
\begin{align*}
\tilde{\nu}_{\delta}=\sum_{k=1}^p a_{k,\delta}\chi_k.
\end{align*}
\item[\textbf{7.)}] Set $\mu=\mu+\tilde{\nu}_{\delta}$, $\lambda=\lambda_0+\tau\tilde{\nu}_{\delta}$.
\end{itemize}
\noindent
}

\subsubsection*{Exact Data}

We start with exact data, i.e. data without noise and due to the definition of $\delta$ given in (\ref{def_delta}),
with $\delta=0$.

	{\color{black}
\begin{remark}
Performing the single implementation steps on a laptop with ${\it 2\times 2.5}$ GHz and 8 GB RAM, we obtained the following computation times:
Step \textbf{1.)}, i.e., the determination of the matrix ${\bf V}$, was done in  9 min 1 s. The Fr\'echet derivative is computed in 53 s in steps  \textbf{2.)}-\textbf{4.)}.
The solution of the minimization problem (step \textbf{5.)}-\textbf{7.)}) is calculated in 6 min 27 s.
\end{remark}
}

\noindent
\\
Figure \ref{3d_mon_bas_mod} presents the results as 3D plots, while Figure \ref{cut_mon_bas_mod_mu}
shows the corresponding cuts for $\mu$. For the same 
reasons as discussed in Section $3$, we change the transparency of the plots of the $3$D reconstruction
of Figure \ref{3d_mon_bas_mod} as indicated in Figure \ref{alpha_mu_mon}.
Thus, tetrahedrons
with low values have a higher transparency, whereas tetrahedrons with large values are plotted opaque.
  \begin{figure}[H]
 \begin{center}
  \includegraphics[width=0.35\textwidth]{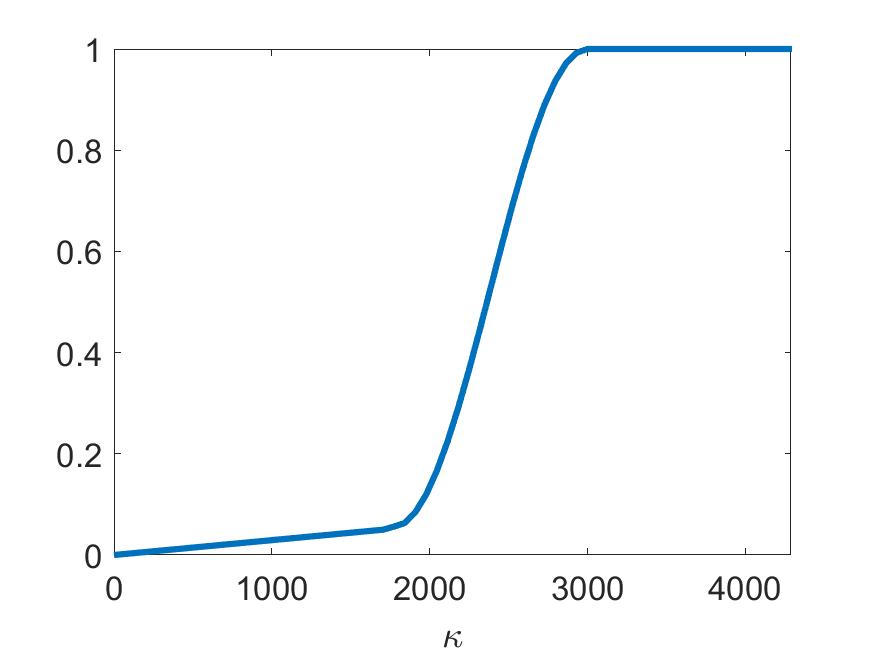}
 \caption{Transparency function for the plots in Figure \ref{3d_mon_bas_mod} mapping the values of $\kappa$ to $\alpha(\kappa)$.}\label{alpha_mu_mon}
 \end{center}
  \end{figure}
  \noindent
  Figures \ref{3d_mon_bas_mod} - \ref{cut_mon_bas_mod_mu} show that solving the minimization problem (\ref{min_prob_disc}) indeed
yields a detection and reconstruction with respect to both Lam\'{e} parameters $\mu$ and $\lambda$. 

  \begin{figure}[H]
 \begin{center}
  \includegraphics[width=0.85\textwidth]{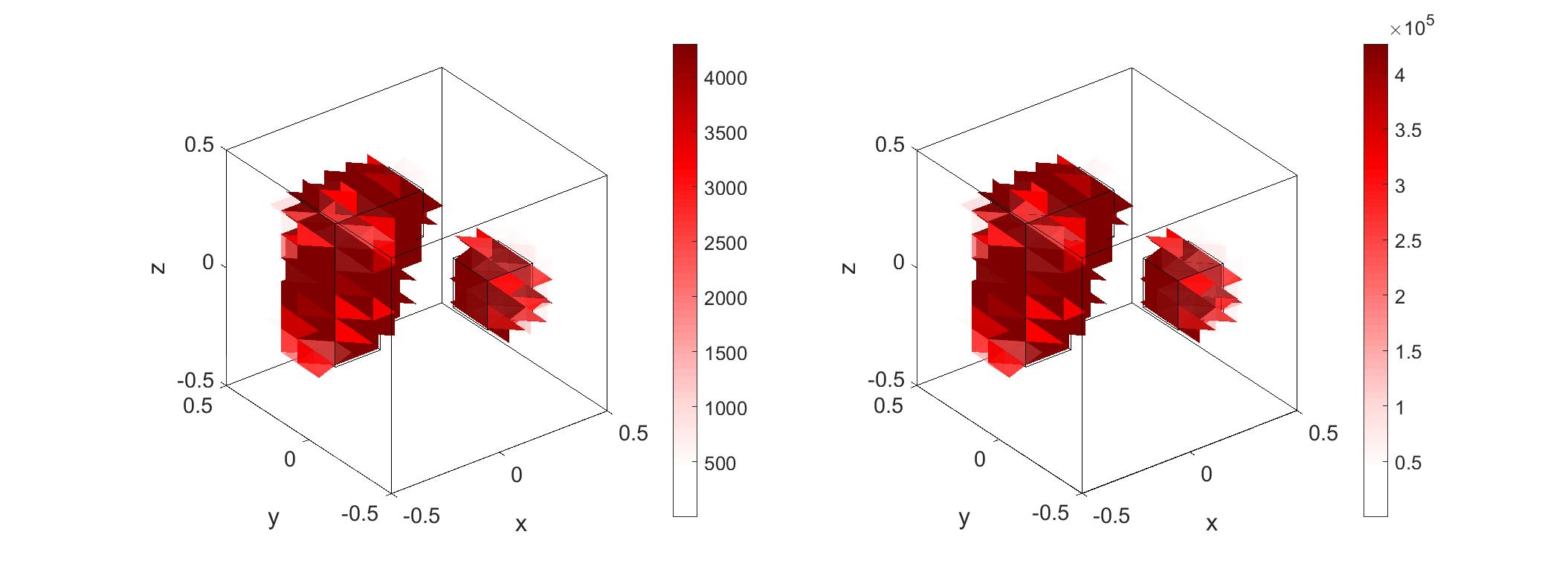}
 \caption{Shape reconstruction of two inclusions (red) of the reconstructed difference in the 
 Lam\'{e} parameter $\mu$ (left hand side) and $\lambda$ (right hand side) without noise, $\delta=0$
and transparency function $\alpha$ as shown in Figure \ref{alpha_mu_mon}.}\label{3d_mon_bas_mod}
 \end{center}
  \end{figure}

 \begin{figure}[H]
 \begin{center}
  \includegraphics[width=0.32\textwidth]{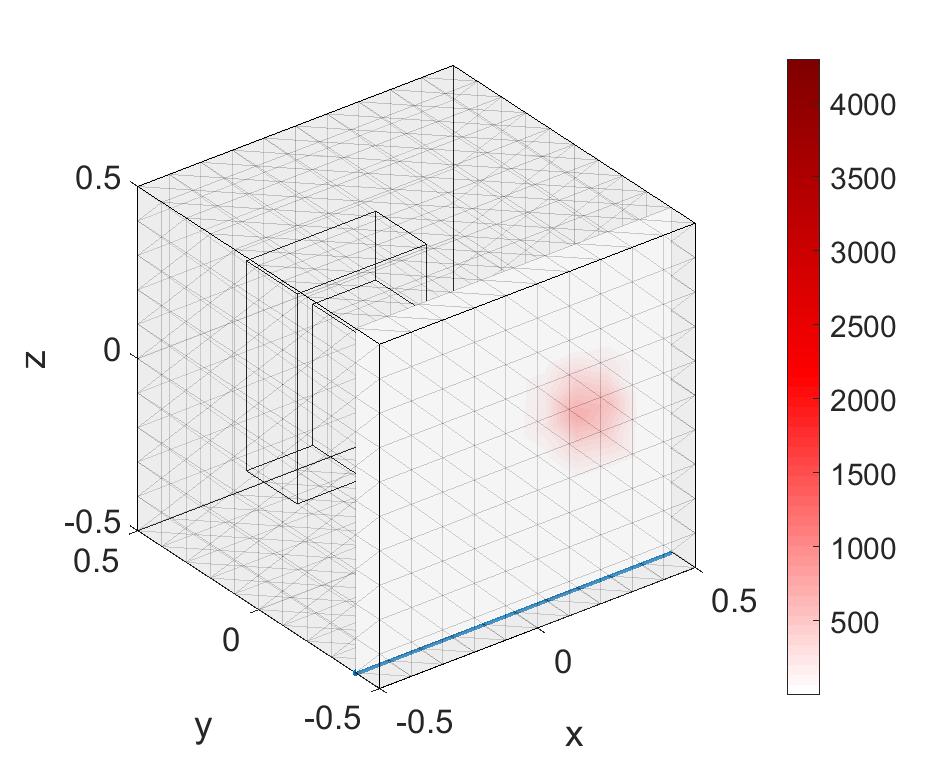}
  \includegraphics[width=0.32\textwidth]{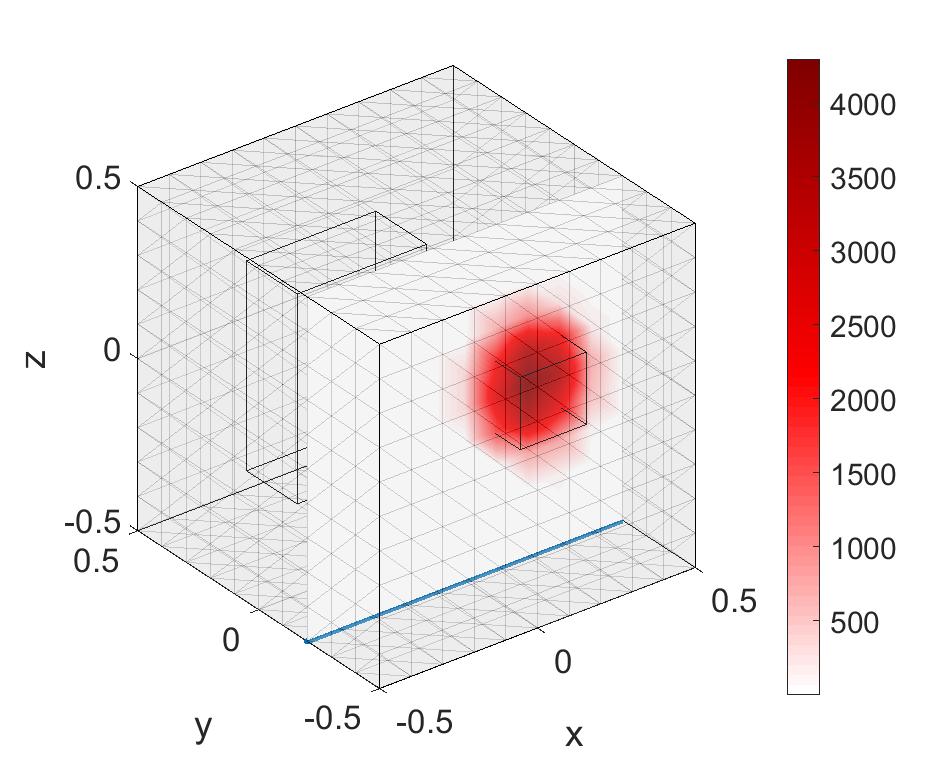}\\  
  \includegraphics[width=0.32\textwidth]{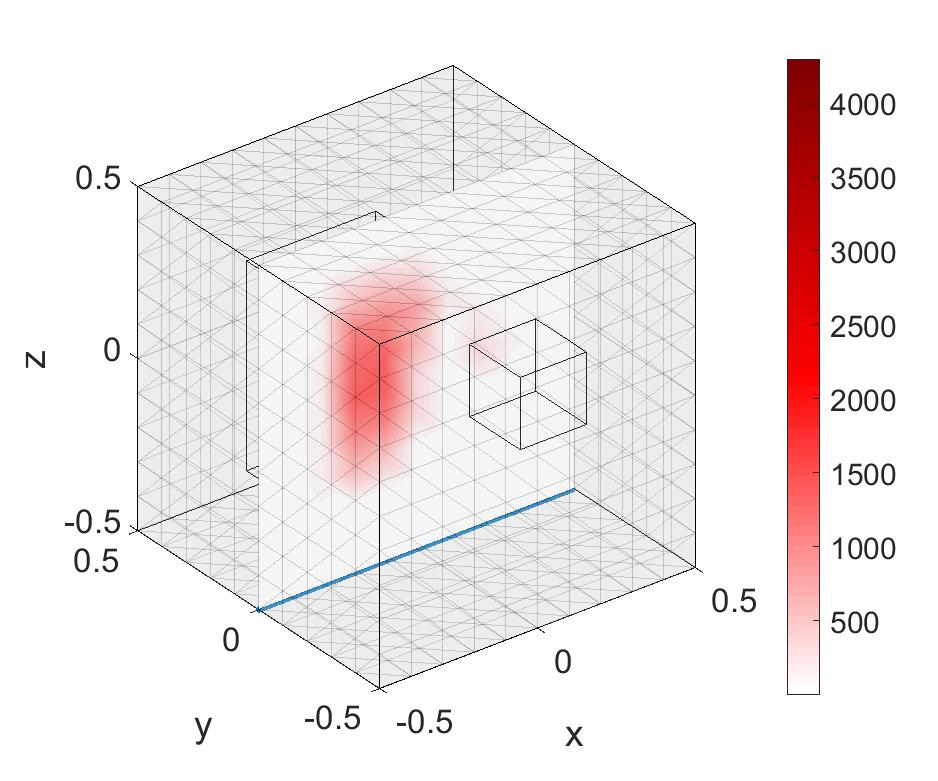}\
   \includegraphics[width=0.32\textwidth]{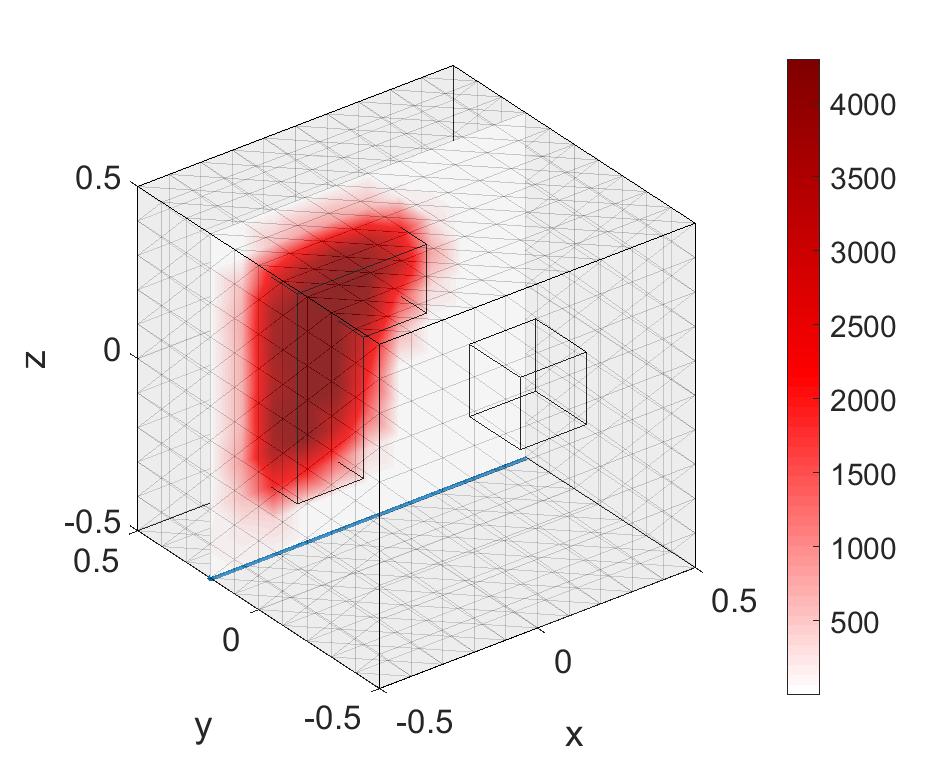}
    \includegraphics[width=0.32\textwidth]{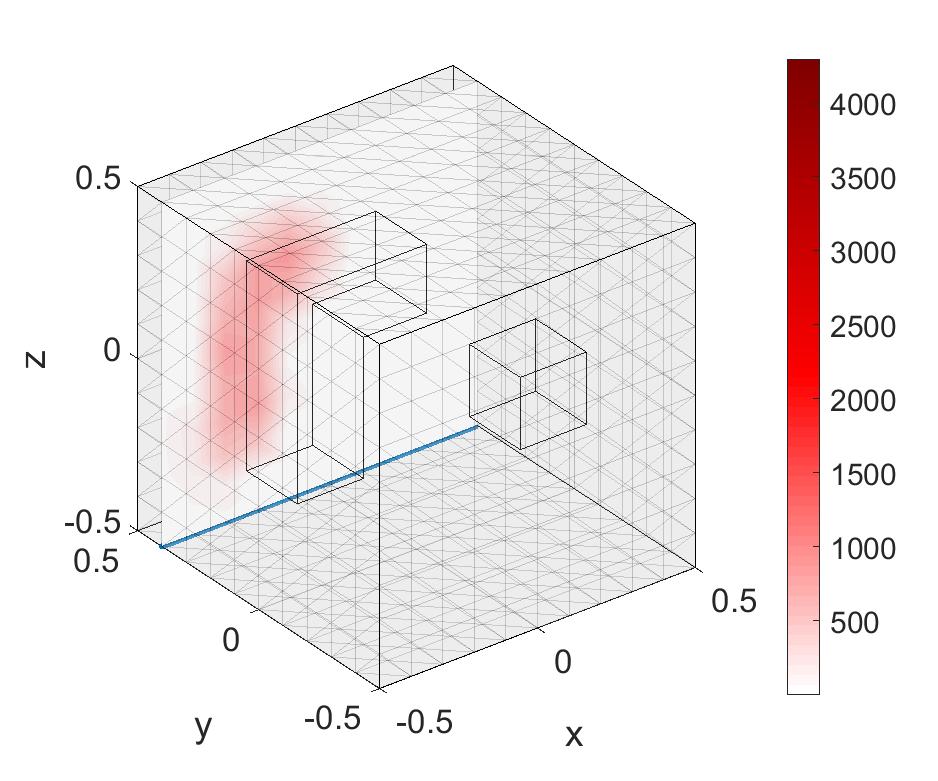}
  \caption{Shape reconstruction of two inclusions (red) of the reconstructed difference in 
  the Lam\'{e} parameter $\mu$ depicted as cuts without noise and $\delta=0$.}\label{cut_mon_bas_mod_mu}
  \end{center}
  \end{figure}

\begin{remark}
Compared with the results obtained with the one-step linearization method as depicted in Figure \ref{one_step_suitable_3d} (right hand side), Figure \ref{3d_mon_bas_mod} shows an improvement because we are now able to also obtain information concerning $\lambda$
which is not possible with the heuristic approach considered in (\ref{min_stand}).
\end{remark}

\subsubsection*{Noisy Data} 

Finally, we take a look at noisy data with a relative noise level $\eta=10\%$, where the $\delta$ is
determined as given in (\ref{def_delta}).
\\
\\
Figures \ref{3d_mon_bas_mod_noise} - \ref{cut_mon_bas_mod_mu_noise} document that
we can even reconstruct the inclusions for noisy data which is a huge advantage compared with the 
results of the one-step linearization (see Figure \ref{one_step_suitable_3d_noise}-
\ref{one_step_suitable_cuts_mu_noise}). This shows us, that the numerical simulations based on
the monotonicity-based regularization are only marginally affected by noise as we have proven in theory,
e.g., in Theorem \ref{theo_conv_+}.

 \begin{figure}[H]
 \begin{center}
  \includegraphics[width=0.85\textwidth]{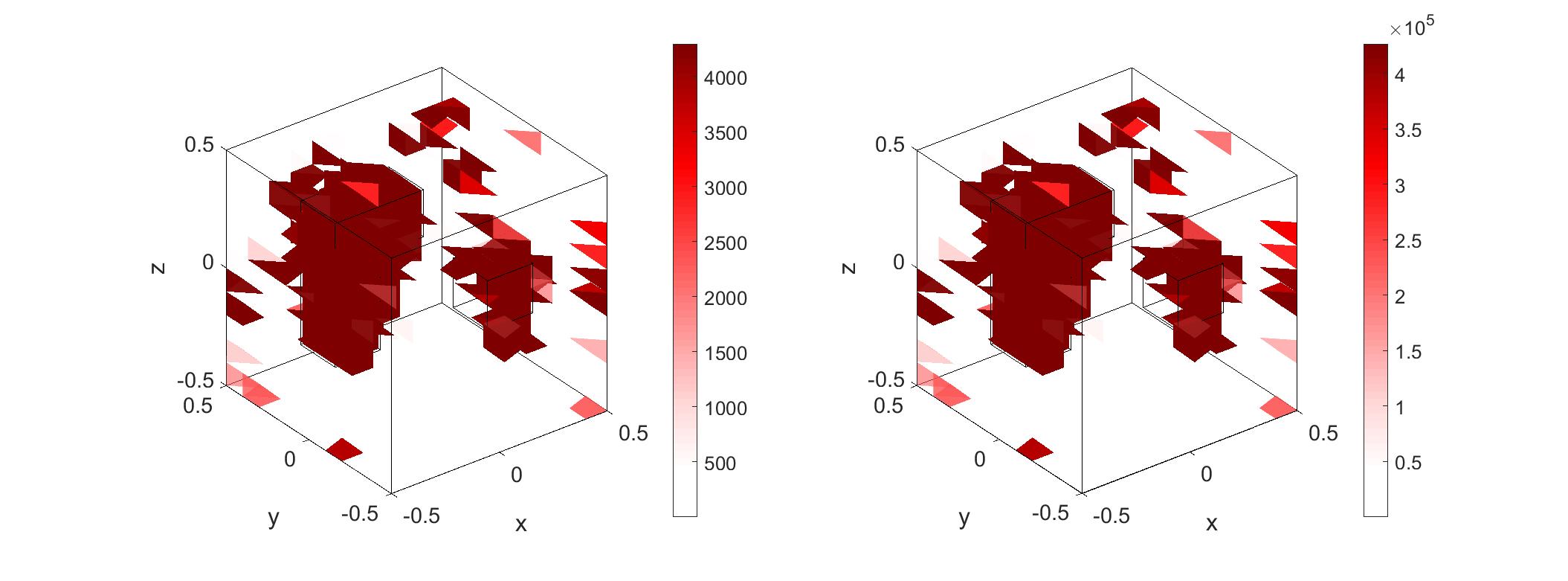}
 \caption{Shape reconstruction of two inclusions (red) of the reconstructed difference in the 
 Lam\'{e} parameter $\mu$ (left hand side) and $\lambda$ (right hand side) 
 with relative noise $\eta=10\%$, $\delta=8.3944\cdot 10^{-8}$
 and transparency function $\alpha$ as shown in Figure \ref{alpha_mu_mon}.}\label{3d_mon_bas_mod_noise}
 \end{center}
  \end{figure}

 \begin{figure}[H]
 \begin{center}
  \includegraphics[width=0.32\textwidth]{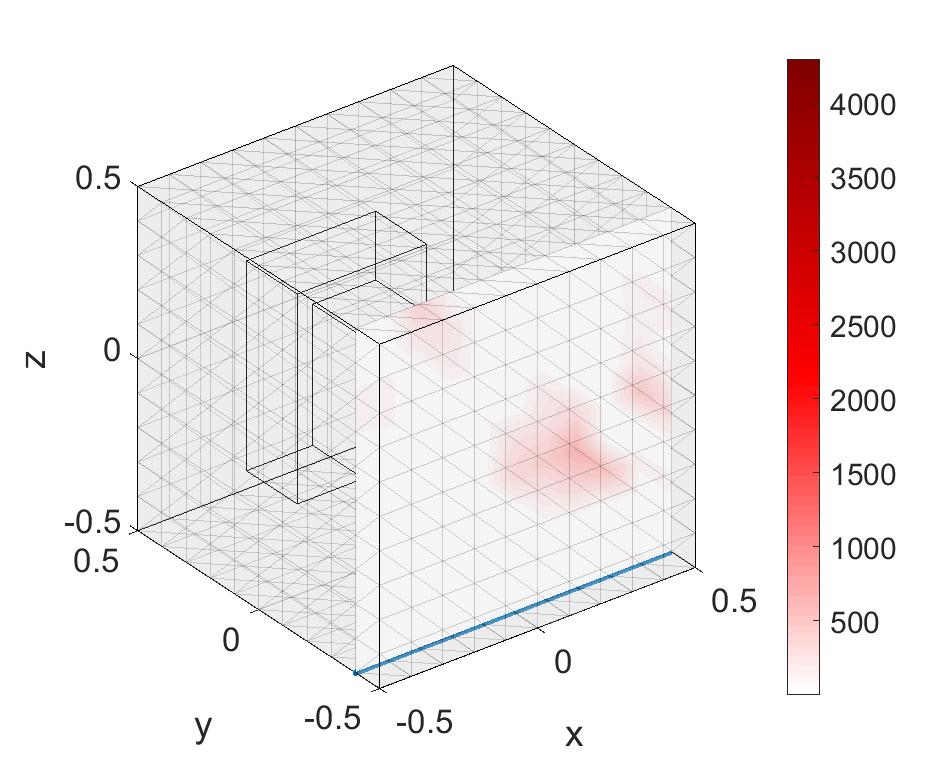}
  \includegraphics[width=0.32\textwidth]{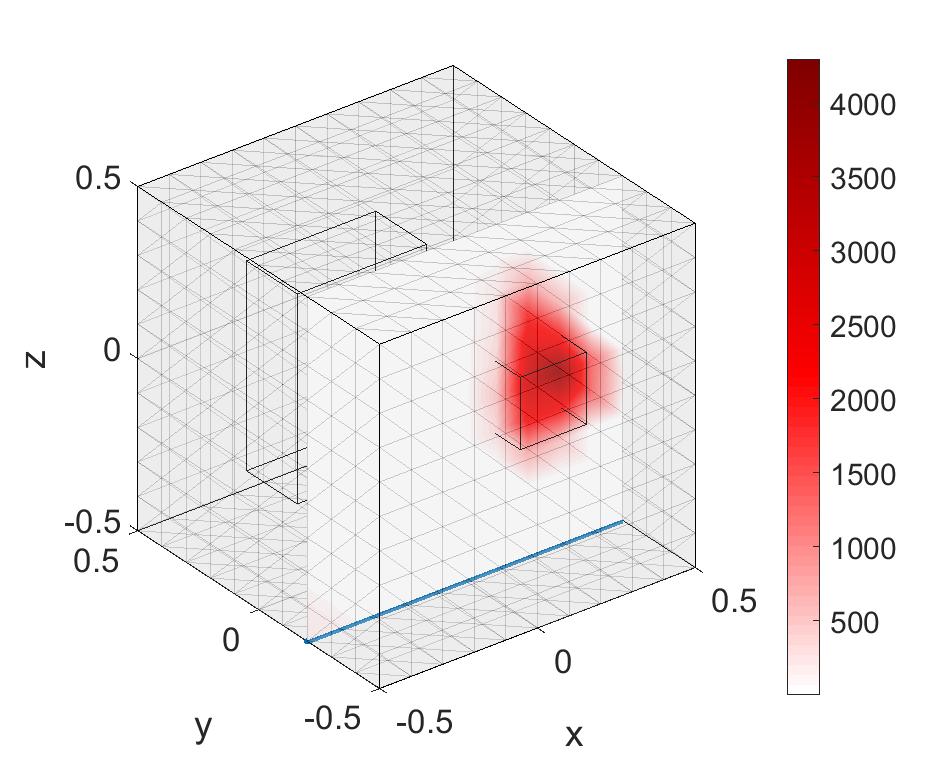}\\
    \includegraphics[width=0.32\textwidth]{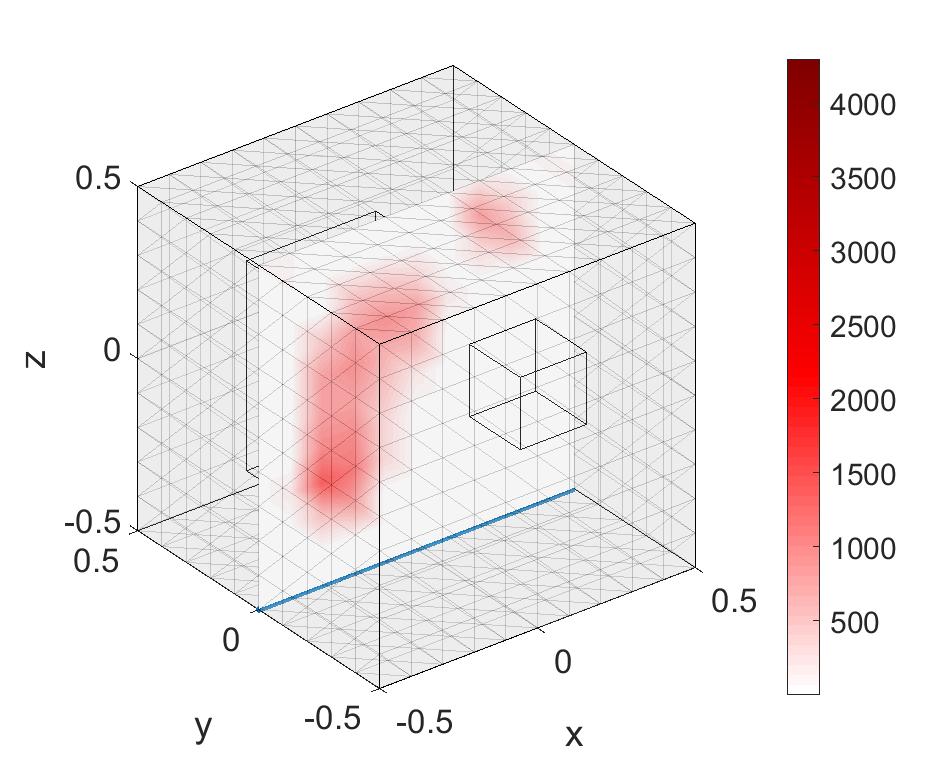}
  \includegraphics[width=0.32\textwidth]{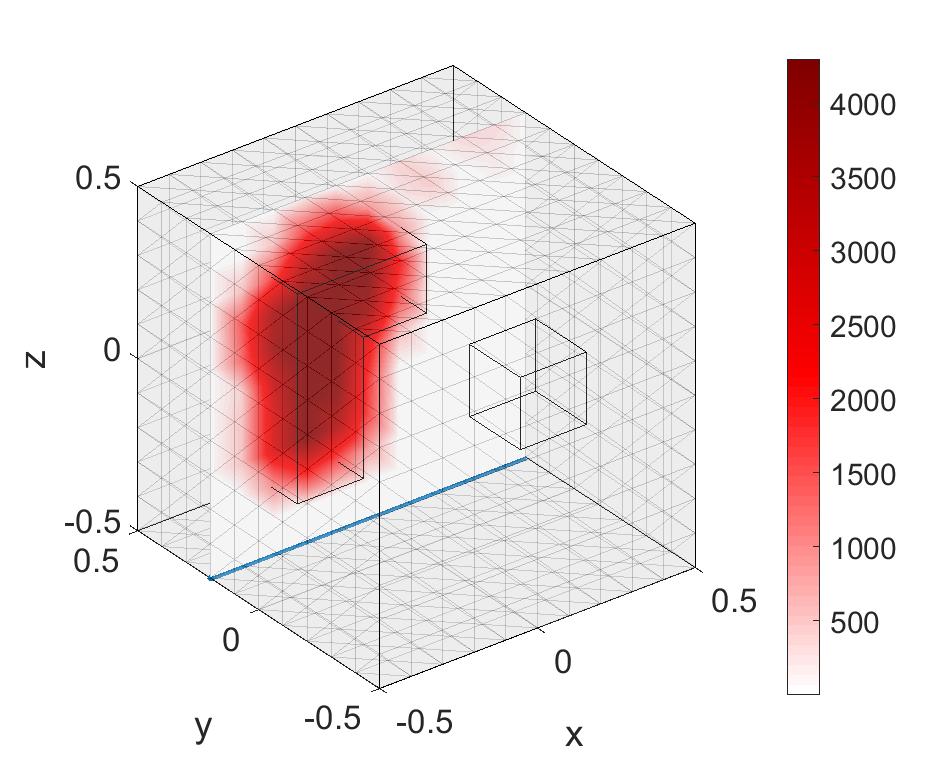}
  \includegraphics[width=0.32\textwidth]{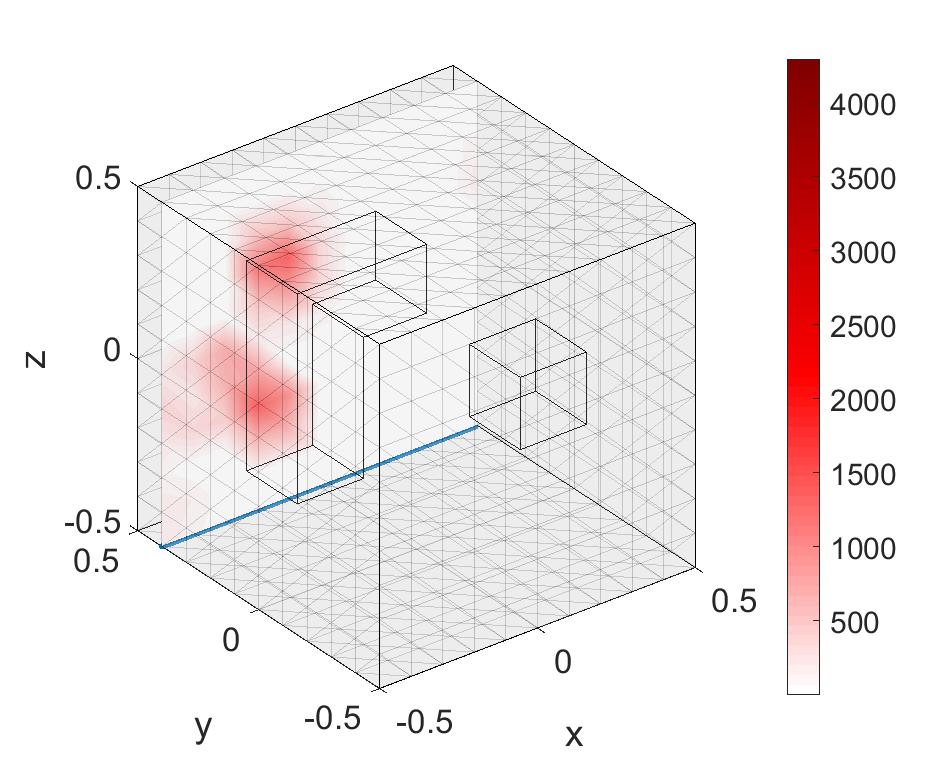}
  \caption{Shape reconstruction of two inclusions (red) of the reconstructed difference in 
  the Lam\'{e} parameter $\mu$ depicted as cuts with relative noise $\eta=10\%$ and $\delta=8.3944\cdot 10^{-8}$.}\label{cut_mon_bas_mod_mu_noise}
  \end{center}
  \end{figure}

\section{Summary}

In this paper we introduced a standard one-step linearization method applied to the Neumann-to-Dirichlet 
operator as a heuristical approach and a monotonicity-based regularization for solving the resulting minimization problem. 
In addition, we proved the existence of such a minimizer. Finally, we presented numerical examples.

\section*{Appendix}
{\color{black}

For the monotonicity-based regularization we focused on the case $\lambda \geq \lambda_0$, $\mu \geq \mu_0$ (see 
Section $5$). For sake of completeness, we formulate the corresponding results for the case that $\lambda \leq \lambda_0$, $\mu \leq \mu_0$.
  Thus, we summarize the corresponding main results and define the set
   {\color{black}
 \begin{align*}
 \mathcal{C}:=
\bigg\lbrace\nu\in L^\infty_+(\Omega)&:\nu=\sum_{k=1}^p a_k \chi_{k},\,a_k\in\mathbb{R},\,0\geq a_k\geq-\min(a_{\max},\beta_k)\bigg\rbrace
 \end{align*}
\noindent
where the quantities $a_{\max}$ and $\tau_{\max}$ are defined as
 \begin{align}
 a_{\max}&:=c^\mu\quad \text{and} \quad
\tau:=\frac{c^\lambda}{c^\mu},
 \end{align}
\noindent
such that
\begin{align}\label{cond_a_minus}
-2\left(\mu-\mu_0\right)+2a&\geq 0,\\
 -\left(\lambda-\lambda_0\right)+\tau a&\geq 0\label{cond_tau_minus}
\end{align}
  \noindent
   for all $0\geq a\geq - a_{\max}$.
   \\
   \\
  \begin{remark}
   The value $a_{\max}$ is obtained from the estimates in Lemma \ref{mono} 
   which results in a different upper bound $a$ compared with the case $\lambda \geq \lambda_0$, $\mu \geq \mu_0$.
  \end{remark}
  }
  \noindent
  \\
  Thus, the theorem for exact data is given by
  \\
  \begin{theorem}\label{theo_min_prob_minus}
 Consider the minimization problem
 \begin{align}\label{min_prob_minus}
\min_{\nu\in\mathcal{C}}\Vert {\bf R}(\nu)\Vert_F.
 \end{align}
 \noindent
 The following statements hold true:
 \begin{itemize}
 \item[(i)] Problem (\ref{min_prob_minus}) admits a unique minimizer $\hat{\nu}$.
 \item[(ii)] $\mathrm{supp}(\hat{\nu})$ and $\mathcal{D}$ agree up to the pixel partition, i.e. for any pixel $\mathcal{B}_k$ 
 \begin{align*}
 \mathcal{B}_k\subset\mathrm{supp}(\hat{\nu})\quad\text{if and only if}\quad \mathcal{B}_k\subset \mathcal{D}.
 \end{align*}
 \noindent
 Moreover, 
 \begin{align*}
\hat{\nu}=\sum_{\mathcal{B}_k\subseteq\mathcal{D}}a_{\max}\chi_{k}.
 \end{align*}
 \end{itemize}
 \end{theorem}
  \noindent
 The corresponding results for noisy data is formulated in the following theorem, 
 where ${\bf R}_\delta(\nu)$ represents the matrix $(\langle  g_i,r_\delta(\nu) g_j\rangle)_{i,j=1,\ldots,M}$ and the admissible set for noisy data is defined by
 \begin{align*}
 \mathcal{C}_\delta:=\bigg\lbrace \nu\in L^\infty_+(\Omega):
\nu=\sum_{k=1}^p a_k \chi_{k},\,a_k\in\mathbb{R},\,0\geq a_k\geq-\min(a_{\max},\beta_{k,\delta})\bigg\rbrace.
 \end{align*}
 
   \begin{theorem}\label{theo_conv_minus}
 Consider the minimization problem
 \begin{align}\label{min_prob_delta_minus}
\min_{\nu\in\mathcal{C}_\delta}\Vert {\bf R}_\delta(\nu)\Vert_F.
 \end{align}
 \noindent
 The following statements hold true:
 \begin{itemize}
 \item[(i)] Problem (\ref{min_prob_delta_minus}) admits a minimizer.
 \item[(ii)]  Let $\hat{\nu}=\sum\limits_{\mathcal{B}_k\subseteq\mathcal{D}}a_{\max}\chi_k$
 be the minimizer of (\ref{min_prob_minus}) and  
 $\hat{\nu}_\delta=\sum\limits_{k=1}^{p} a_{k,\delta}\chi_k$
 of problem (\ref{min_prob_delta_minus}), respectively. 
 Then $\hat{\nu}^\delta$ converges pointwise and uniformly to $\hat{\nu}$ as $\delta$ goes to $0$.
 \end{itemize}
 \end{theorem}
 }

\bibliographystyle{plain}
\bibliography{references}
\end{document}